\documentclass{SCAEOL}%SCAEOL for online version; SCAE for publication version; SCAES for the paper dedicated to somebody.
\numberwithin{equation}{section}
\usepackage{tkz-euclide}

\newcommand{\RR}{{\mathbb R}}

\newcommand{\vv}{{\mathbf v}}
\newcommand{\cc}{{\mathbf c}}

\usepackage{float}
\restylefloat{figure}
\begin{document}

%Basic Information
\Year{2013} %
\Month{January}
\Vol{56} %
\No{1} %
\BeginPage{1} %
\EndPage{XX} %
\AuthorMark{Dehbi L  {\it et al.}}
%\ReceivedDay{November 17, 2012}
%\AcceptedDay{January 22, 2013}
%\PublishedOnlineDay{; published online January 22, 2013}
%\DOI{10.1007/s11425-000-0000-0} % The author doesn't need fill in it.

% \title[short text for running head]{full title}{comments for title}
\title{Algebraic Curve Interpolation for Intervals via Symbolic-Numeric Computation}{}
%\title{Implicit Interpolation for Intervals via Symbolic-Numeric Computations}{}

% \author[]{Full name}{footnote}
% Remark:  One \author for one author

\author[1]{Lydia Dehbi}{}
\author[1]{Zhengfeng Yang}{Corresponding author}
\author[1]{Chao Peng}{}
\author[2]{Yaochen Xu}{}
\author[2]{Zhenbing Zeng}{}

\address[{\rm1}]{School of Software Engineering, East China Normal University, Shanghai {\rm 200062}, China;}
\address[{\rm2}]{Department of Mathematics, Shanghai University, Shanghai {\rm 200444}, China;}
\Emails{dehbilydia@sei.ecnu.edu.cn,
zfyang@sei.ecnu.edu.cn, cpeng@sei.ecnu.edu.cn, xuyaochen@sibcb.edu.cn,
zbzeng@shu.edu.cn}\maketitle

%     Abstract is required.

 {\begin{center}
\parbox{14.5cm}{\begin{abstract}
Algebraic curve interpolation is described by specifying the location of
$N$ points in the plane and constructing an algebraic curve of a function
$f$ that should pass through them. In this paper, we propose a novel approach to construct the algebraic curve that interpolates a set of data (points or neighborhoods).
This approach aims to search the polynomial with the smallest degree interpolating the given data. Moreover, the paper also presents an efficient method to reconstruct the algebraic curve of integer coefficients with the smallest degree and the least monomials that interpolates the provided data. The problems are converted into optimization problems and are solved via Lagrange multipliers methods and symbolic computation. Various examples are presented to illustrate the proposed approaches.

\vspace{-3mm}
\end{abstract}}\end{center}}

%  Keyword is required.

%  \subjclass is required.

%%%%%%%%%%%%%%%%%%%%%%%%%%%%%%%%%%%%%%%%%%%%%%%%%%%%%%%%%%%%
\renewcommand{\baselinestretch}{1.2}
%\begin{center} \renewcommand{\arraystretch}{1.5}
%{\begin{tabular}{lp{0.8\textwidth}} \hline \scriptsize
%{\bf Citation:}\!\!\!\!&\scriptsize First1 L N, First2 L N, First3 L N.  SCIENCE CHINA Mathematics  journal sample. Sci China Math, 2013, 56, doi: 10.1007/s11425-000-0000-0\vspace{1mm}
%\\
%\hline
%\end{tabular}}\end{center}

%%%%%%%%%%%%%%%%%%%%%%%%%%%%%%%%%%%%%%%%%%%%%%%%%%%%%%%%%%%%
%% Text of article.
%%%%%%%%%%%%%%%%%%%%%%%%%%%%%%%%%%%%%%%%%%%%%%%%%%%%%%%%%%%%
%    Section headings
\baselineskip 11pt\parindent=10.8pt  \wuhao

\section{Introduction}
Polynomial interpolation is a rather old mathematical problem \cite{gasca2001history} and has significant applications in various practical engineering problems \cite{bagheri2017self,feng2018quasi,nunez2018accelerated,dell2021numerical,dell2021solving,dell2022numerical,meletiou2003cryptography,9805746}. It has become a fundamental research tool in approximation theory and numerical analysis \cite{forsythe1977computer,fausett2003numerical,datta2004numerical}. Compared to univariate polynomial interpolation, the study of multivariate polynomial interpolation and implicit function interpolation is relatively new. In recent years, many scholars have focused on these topics and proposed new algorithms \cite{de1990multivariate,errachid2020rmvpia,essanhaji2022lagrange,gasca2000polynomial,neidinger2009multivariable, neidinger2019multivariate,varsamis2014optimal,varsamis2014newton}.

In the study of general interpolation problems, researchers seek to construct functions of the form $y=f(x_1,x_2,\ldots,x_n)$ that minimize the overall deviation from the given data (such as \cite{varsamis2014optimal,varsamis2014newton}). In contrast to existing research, the problem discussed in this paper is as follows: the given data consists of small neighborhoods in space (such as square or spherical neighborhoods), and the goal of interpolation is to construct polynomial implicit functions that pass through these small neighborhoods. The measure of interpolation quality is not the deviation between the interpolated function and the known data, but rather the properties of the degree and coefficients of the constructed polynomial. For example, the goal is to minimize the degree of the polynomial and ensure that the coefficients are integers and have the smallest possible absolute values. Implicit function interpolation problems, where $f(x_i,y_i) = 0, 1 \leq i \leq N$, seek the function $f(x,y)$ whose curve passes through these points or satisfies $f(x_i,y_i) \approx 0$. Such problems often arise in engineering or experimental data processing, such as constructing empirical formulas based on experimental data (as seen in \cite{Zhang,Gimenez}) or problems in computer graphics that involve reconstructing implicit surfaces from discrete points (see \cite{Thevenaz,Shen,Morse,Anjyo,anjyo2014scattered,smolik2018large, dell2022numerical,drake2022implicit}).  

In previous research, if there were errors in the data, it was assumed that the given data represented small neighborhoods, and the goal was for the interpolated implicit function to pass through the centers of these neighborhoods as closely as possible. In our problem, the size of the data neighborhoods is fixed, and the interpolated function only needs to pass through any point within the neighborhood. The implicit functions obtained directly using existing methods generally do not satisfy our optimization objectives regarding the degree and coefficients of the implicit functions. In practical applications, we often need to fit an empirical formula based on a limited set of data obtained from experiments or observations. For these problems, if the degree of the interpolated polynomial implicit function is too high or the range of coefficient values is too large, the predictive capability of the formula may be reduced. Therefore, it is desirable for the interpolated formula to avoid overfitting and be as simple as possible in terms of its form (degree and coefficients). This increases the likelihood of capturing the essential connections between the relevant variables.

The rest of this paper is organized as follows. 
In the next section, we will first study the computation problem of finding the lowest degree of an implicit polynomial $f(x, y)$ passing through a given set of accurate points $(x_1, y_1), (x_2, y_2), \ldots, (x_N, y_N)$ on the plane. We construct a discriminant matrix $A^{[d]}$ of order $(d+1)(d+2)/2$ and prove the relationship between the lowest degree of the implicit polynomial and the determinant of the discriminant matrix being zero. Then, we investigate the computation problem of finding the lowest degree of a polynomial passing through a given set of rectangular regions $B_1, B_2, \ldots, B_N$ on the plane, extending the discriminant method from accurate points to the case of rectangular neighborhoods.
In section 3, we study related sparse interpolation problems. By analyzing the motion data of 38 celestial bodies in the solar system, including planets, asteroids, and dwarf planets, we search for the minimal number of terms in the lowest degree polynomial curve.
In section 4, we investigate the interval interpolation problem of rational-coefficient or integer-coefficient implicit functions with given monomials. We provide an algorithm to compute a positive number for the interpolation problem of integer-coefficient polynomials, such that among the interpolations with coefficients whose absolute values do not exceed this number, there exists a unique one passing through all the given neighborhoods.
In addition to presenting rigorous proofs for several theorems, the paper also provides concrete computational processes for several meaningful theoretical and practical problems, aiming to assist readers in solving similar problems or extending the methods presented in this paper. Finally, Section 5 concludes the paper with a brief summary.

\section{The lowest degree of approximate implicit interpolation}
\label{alg}
Let's consider a set of points $P_1, P_2, \cdots, P_N$ in the two-dimensional plane $\mathbb{R}^2$. Our objective is to find the simplest polynomial equation, denoted as $f(x, y)$, that has the lowest possible degree and passes through all the given points.
In the case where the coefficients of the polynomial are in $\mathbb{Q}$, we also aim to determine the equivalent polynomial equation with integer coefficients.

First of all, we show the existence of such a polynomial. We define the set $\mathcal{P}$ as
$$
\mathcal{P}:=V(P_1,P_2,\cdots,P_N)=\{f\in \mathbb{R}[x,y]\, |\, f(x_i,y_i)=0, ~~ i = 1,2,\cdots,N\},
$$
and we assume a polynomial template of degree $d$ 
\begin{equation}
f(x,y)=c_0+c_1\cdot x+c_2\cdot y+\cdots+c_{{p}-2}\cdot x^{d-1}y+c_{{p}-1} \cdot y^d,
\label{poly-k-d}
\end{equation}
where $f$ has $p = {d+2\choose 2}$ coefficients. Let $\cc_{d}$ and $\vv_{d}$ be the 
coefficient vector and the monomial vector of $f$, i.e.,
$\cc_{d}=[c_0,c_1,c_2,\ldots,c_{p-1}]^T$ and 
$\vv_d=[1,x,y,x^2,xy,y^2,\ldots,xy^{d-1},y^{d}]^T$, $f$ can be written as $f=\cc_{d}^{T} \, \vv_{d}$.
Let $\vv_{d}(x_i,y_i)$ be  the vector by evaluating the monomial vector $\vv_{d}$ with the given point 
$(x_i,y_i)$.
Then a system of homogeneous linear equations is constructed
as follows:
\begin{eqnarray*}
\cc_{d}^{T}\,\vv_{d}(x_j,y_j)=c_0+c_1\cdot x_j+c_2\cdot y_j+\cdots+c_{p-2} \cdot x_j y_j^{d-1}+c_{p-1}\cdot y_j^d=0, \quad j=1,2,\ldots,N.
\end{eqnarray*}

When $d$ is large enough, e. g., verifying, $d \geq u(N):=\lfloor \sqrt{2N+1/4}-1/2\rfloor\approx \sqrt{2N}-1$, we have the following inequality
$$
2p = (d+1)(d+2)\geq \left(\lfloor \sqrt{2N+\frac{1}{4}}-\frac{1}{2}\rfloor+1\right) 
\left(\lfloor \sqrt{2N+\frac{1}{4}}-\frac{1}{2}\rfloor+2\right)>2N.
$$
Therefore, the system of equations \eqref{poly-k-d} must have a non-zero solution $\cc_{d}^{\ast}=[c_0^{\ast},c_1^{\ast},c_2^{\ast},\ldots,c_{p-1}^{\ast}]^{T}$, and when $(x_i,y_i)\,(1\leq i\leq N)$ is a rational number, each component of the solution vector can also be taken as an integer. This proves the existence of a polynomial of degree $d$:
$$ 
f^{\ast}(x,y):= C \cdot ({\cc_{d}^{\ast}}^{T}\,\vv_{d})=C\cdot (c_0^{\ast}+c_1^{\ast} \cdot x+c_2^{\ast}\cdot y
+ c_{p-2}^{ast} \cdot x y^{d-1}+ +c_{p-1}^{\ast} \cdot y^d)\in \mathcal{P}.
$$
Obviously, when $N$ points are in some special positions, and $d$ smaller than $\lfloor\sqrt{2N}\rfloor-1$ the system of homogeneous equations \eqref{poly-k-d} may also have non-zero solutions, so $f(x,y)$ can be constructed according to \eqref{poly-k}, and the implicit function with smallest degree and passing through the given points can be obtained.

The number of coefficients of an algebraic curve of degree $d$ grows dramatically as $d$
increases. Hence, keeping the degree of a curve in
a reasonable range is very important for fast computation and fewer numerical errors.
In what follows, we present a way to determine the smallest degree of polynomials that pass through a given dataset.

Next, we present a method to determine whether $d$ is the smallest degree based on the coordinates of $P_i$ without solving the system of linear equations. 
Given $N$ points and a positive integer $d$, the interpolation of the implicit function matrix ($d-th$ implicit interpolation aggregate matrix) is defined as follows.
%Note that $d = \lfloor \sqrt{2N}\rfloor$ satisfies the inequality $(d+1)(d+2)>2N$, which implies that
%if $P_i\cap \mathbb{Q}^2\not=\emptyset$ for $i=1,2,\cdots,N$, then there exists $f(x,y)\in \mathbb{Z}[x,y]$ of degree $\lfloor \sqrt{2N}\rfloor$ that
%passes through $P_1,P_2,\cdots,P_N$ simultaneously. Thus, we can take $ \lfloor \sqrt{2N}\rfloor $ as an upper bound of the total degree of the minimal polynomial.
%The lower bound of polynomial $f\in \mathcal{P}$ can be estimated
\begin{definition}\label{def:1}  Given $N$ points ($P_i=(x_i,y_i)$), an integer $d\geq 1$ , and $p={d+2 \choose 2}$, we denote $A^{[d]}$ as a $p \times p$ symmetric matrix defined as $A^{[d]}=\sum_{i=1}^{N}A_i$,
where $A_i=\vv_d(x_i,y_i) \vv_d(x_i,y_i)^{T}$, that is,   
$$
A_i=[1,x_i,y_i,\cdots,x_i^d,x_i^{d-1}y_i,\cdots,y_i^d]^T \, [1,x_i,y_i,\cdots,x_i^d,x_i^{d-1}y_i,\cdots,y_i^d].
$$
\end{definition}
We have the following theorem.
\begin{theorem}
Given $N$ points and an integer $d$, if the curve $f(x,y)=0$ passes through the given points, then
\begin{equation}
\min_{f\in \mathcal{P}} ~~ \mbox{\rm degree} ~(f)\leq d  ~~ \Leftrightarrow ~~\det (A^{[d]}) = 0.
\label{min-degree}
\end{equation}
\label{main-theorem}
\end{theorem}
\begin{proof} \quad
 Let us first prove $ \Leftarrow$.

Assuming $(x_1,y_1),\cdots,$ $(x_N,y_N)$ satisfy
$ \det(A^{[d]})=0$,
we have the homogeneous linear equation system
$$
A^{[d]} \, \cc_{d} =[0,0,\cdots,0]^T.
$$
This equation has a non-zero solution $\cc_{d}^{\ast}=[c_0^{\ast},c_1^{\ast},\cdots,c_{p-1}^{\ast}]^{T}\in \mathbb{R}^{p}\setminus\{[0,0,\cdots,0]^{T}\}$, leading to:
$$
\sum_{i=1}^{N}A_{i} \cc_{d}^{\ast}=[0,0,\cdots,0]^T.
$$
Thus, we can derive the following expression:
\begin{eqnarray}
&&\sum_{i=1}^{N}(c_0^{\ast}+c_1^{\ast}x_i+c_2^{\ast}y_i+\cdots+c_{p-1-d}^{\ast}x_i^d+\cdots+c_{p-1}^{\ast}y_i^d)^2\nonumber\\ [7pt]
&=&\sum_{i=1}^{N}({\cc_{d}^{\ast}}^{T} \, \vv_{d}(x_i,y_i))(\vv_d(x_i,y_i)^{T} \, \cc_{d}^{\ast}) = \sum_{i=1}^{N} {\cc_d^{\ast}}^{T} 
\,(\vv_{d}(x_i,y_i)\, \vv_d(x_i,y_i)^{T} \, \cc_{d}^{\ast})
=\sum_{i=1}^{N} {\cc_d^{\ast}}^{T} \, (A_{i} \cc_d^{\ast})
\nonumber\\ [7pt]
&=&\sum_{i=1}^{N} {\cc_{d}^{\ast}}^{T} \, [0,0,\cdots,0]^T ~~ = 0,
\end{eqnarray}
which implies that
$$
f(x_i,y_i) = c_0^{\ast}+c_1^{\ast}x_i+c_2^{\ast}y_i+\cdots+c_{p-1-d}^{\ast}x_i^d+\cdots+c_{p-1}^{\ast}y_i^d=0, ~~ i=1,2,\cdots,N.
$$
Therefore, we can conclude that $f(x,y) \in \mathcal{P}$, satisfying the degree constraint~(\ref{min-degree}).

To prove the reverse implication, we Suppose that
$f(x,y)$ is a polynomial of degree $d$ passing through the points $(x_i,y_i)$ for $i=1,2,\cdots,N$, with unknown coefficients $c_0,\ldots,c_{p-1}$.
%and $f(x,y)=\cc_{d}^{T} \vv_{d}$ with unknown coefficients $c_0,\ldots,c_{p-1}$.
%Remark that $f(x,y)$ is affine with respect to the parameters $c_0,c_1,\cdots,c_{p-1}$.
Consider the quadratic optimization problem:
\begin{eqnarray*}\label{BMI2}
   \left\{\begin{array}{ccl}
      &{\displaystyle \min_{\cc_{d} \in \RR^{p}} }  & 
        {\displaystyle \sum_{i=1}^{N}} (\vv_d(x_i,y_i)^T \, \cc_{d})^2 \\
       & s.t. & \|\cc_{d}\|_{2}^{2}=1. 
       % &     & Z\succeq 0.
   \end{array}\right.
\end{eqnarray*}

we can use the Lagrange multiplier method and the necessary optimality conditions to derive the following equality:
$$
\sum_{i=1}^N\left(\hat{c}_0+\hat{c}_1 x_i+\cdots+\hat{c}_{p-1} y_i^d\right)^2+\lambda(\hat{c}_0^2+\hat{c}_1^2+\cdots+\hat{c}_{p-1}^2)=0,
$$

From this equation, together with the fact that $(x_i,y_i)\in Z(f)$ and
$[\hat{c}_0,\hat{c}_1,\cdots,\hat{c}_{p-1}]^{T}\in \mathbb{R}^p\setminus\{[0,0,\cdots,0]^{T}\},$
we can conclude that $\lambda=0$.
Thus, $[\hat{c}_0,\hat{c}_1,\cdots,\hat{c}_{p-1}]^{T}$ is a non-zero solution of the following homogeneous linear equations
\begin{eqnarray*}
&&\sum_{i=1}^{N} (c_0+x_ic_1+y_i c_2+\cdots+x_i^d c_{p-1-d}+\cdots+y_i^d c_{p-1}) = 0,\\
&&\sum_{i=1}^{N} x_i(c_0+x_ic_1+y_i c_2+\cdots+x_i^d c_{p-1-d}+\cdots+y_i^d c_{p-1}) = 0,\\
&&\vdots\\
&&\sum_{i=1}^{N} y_i^d(c_0+x_ic_1+y_i c_2+\cdots+x_i^d c_{p-1-d}+\cdots+y_i^d c_{p-1}) = 0.
\end{eqnarray*}
In matrix form, this can be written as:
$$
\left[
\begin{bmatrix}
1\\ x_1\\ \vdots \\ y_1^d
\end{bmatrix}
[1,x_1,\cdots,y_1^d]
+
\cdots
+
\begin{bmatrix}
1\\ x_N\\ \vdots \\ y_N^d
\end{bmatrix}
[1,x_N,\cdots,y_N^d]
\right]
\;
\begin{bmatrix}
c_0\\ c_1\\ \vdots \\ c_{p-1}
\end{bmatrix}
=(\sum_{i=1}^{N}A_i) \cdot \begin{bmatrix}
c_0\\ c_1\\ \vdots \\ c_{p-1}
\end{bmatrix}
=
\begin{bmatrix}
0\\ 0\\ \vdots \\ 0
\end{bmatrix}
$$
which can be further simplified as:
$$
A^{[d]} \, [c_0,c_1,\cdots,c_{p-1}]^T
=[0,0,\cdots,0]^T.
$$
Thus, we have 
$
\det (A^{[d]}) =0.
$
%AS CLaimed in Theorem~\ref{main-theorem}.
\end{proof}
\begin{corollary} 
For any integer $d$ and any points $(x_1,y_1),\cdots,(x_q,y_q)\in \mathbb{R}^2$ with $q< {{d+2}\choose 2}$, we have
$$
\det\left(
\sum_{i=1}^{q} A_i \right)=0.
$$
\label{akeq0}
\end{corollary}
\begin{proof}
From Definition~\ref{def:1},
$\mbox{ \rm rank}(A_i)=1$ for each $i$, which implies
$$
\mbox{\rm rank}(\sum_{i=1}^{q} A_i)\leq \sum_{i=1}^q \mbox{\rm rank}(A_i)=q.
$$
Since $q<(d+1)(d+2)/2$, it leads to $\det \left(\sum_{i=1}^{q}A_i\right)=0$.
\end{proof}
\vspace{0.5cm}
%%%%%%%%%%%%%%%%%%%%%%%%%%%%%%%%%%%%%%%%%%%%%%%%%%%%%%%%%%%%%%%%%%%%%%%
The following result provides a lower bound on the degree of polynomials in $V(P_1,P_2,\cdots,P_N)$, based on the well-known B\'{e}zout theorem (cf.\cite{walker1950algebraic}).
\begin{corollary}\quad
For any points $P_1=(x_1,y_1),\cdots,P_N=(x_N,y_N)\in \mathbb{R}^2$.
If there exists a polynomial $f\in \mathbb{R}[x,y]$ of degree $d$ that passes through $(x_{1},y_{1}),\cdots,(x_q,y_{q})$ but not through $(x_{q+1},y_{q+1}), \cdots,$ $(x_N,y_N),$ where
$
{d+2\choose 2}-1\leq q < N,
$
then
$$\min_{f\in \mathcal{P}}  ~~ \mbox{\rm degree}~(f)\geq ~ \max~(d+1, ~\lceil\frac{q}{d}\rceil).
$$
\label{cor-bezout}
\end{corollary}
In other words, if we find a subset of $q$ points in $P_1,P_2,\cdots,P_N$ that lie on a straight line, then it implies that either all the points $P_1,P_2,\cdots,P_N$ lie on that same line, or any polynomial passing through $P_1,P_2,\cdots,P_N$ must have a degree of at least $q$.

The following example explain the usage of Theorem \ref{main-theorem} and Corollary \ref{akeq0}, \ref{cor-bezout}

\begin{example}\label{ex1}
%We illustrate our approach using the following example. Consider the following 22 points in the plane as shown in Figure~\ref{17points}.
Consider the following example to illustrate our approach. We have 22 points in the plane, as shown in Figure~\ref{17points}. The coordinates of these points are given in Table \ref{tab1}.

\begin{table}[!h]
\caption{The coordinates of the $22$ points.}{}
\label{tab1}
\begingroup
\setlength{\tabcolsep}{4pt} % Default value: 6pt
\renewcommand{\arraystretch}{1.3} % Default value: 1
$$\begin{array}{|c|c|c||c|c|c||c|c|c|}
\hline
\phantom{PPP}& \phantom{xxxx}x\phantom{xxxx}&\phantom{yy}y\phantom{yy} &
\phantom{PPP}&\phantom{xxxx}x\phantom{xxxx}&\phantom{yy}y\phantom{yy} &
\phantom{PPP}&\phantom{xxxx}x\phantom{xxxx}&\phantom{yyyi}y\phantom{yyyi} \\ \hline
P_1 & 0& 1& P_{9} &1&2 & P_{17}&\frac{\sqrt {18+2\,\sqrt {73}}}{6}&\frac{2}{3}\\
\hline
P_2 &-\frac{2}{3}&\frac{1}{3} & P_{10}&\frac{2\,\sqrt {5}}{3} & \frac{5}{3}&P_{18}&-\frac{\sqrt {9-\sqrt {73}}}{3}&\frac{4}{3} \\
\hline
P_3 &\frac{2}{3} & \frac{1}{3} & P_{11}&\frac{\sqrt{10}}{6}&\frac{5}{3}&P_{19} &\frac{\sqrt {9+\sqrt {73}}}{3} &\frac{4}{3} \\
\hline
P_4 & -\frac{\sqrt{2}}{6} & \frac{1}{3}  & P_{12} &\frac{\sqrt {12-6\,\sqrt {2}}}{8}&\frac{1}{4}& P_{20}&-\frac{\sqrt {9+\sqrt {73}}}{3}&\frac{4}{3} \\
\hline
P_{5} & -\frac{2\,\sqrt {5}}{3} & \frac{5}{3} & P_{13} &\frac{\sqrt {12+6\,\sqrt {2}}}{8}& \frac{1}{4} & P_{21}&-\frac{\sqrt {12+9\,\sqrt {2}}}{4}& 1+\frac{3 \sqrt{2}}{4}  \\
 \hline
P_{6}& -\sqrt{2}& 2 & P_{14}&-\frac{\sqrt {18-6\,\sqrt {7}}}{4}&\frac{3}{2} &P_{22} & \frac{1}{4}& 1+\frac{3 \sqrt{2}}{4} \\
\hline
P_7 & -\frac{\sqrt{6}}{2}& 1 &  P_{15}&\frac{\sqrt {18-6\,\sqrt {7}}}{4}&\frac{3}{2} & & & \\ \hline
P_{8}  & \frac{\sqrt{6}}{2} & 1 & P_{16} & -\frac{\sqrt {18-2\,\sqrt {73}}}{6}&\frac{2}{3} & & &\\\hline
\end{array}$$
\endgroup
\end{table}

We aim to compute the minimal polynomial ${f}(x,y)=0$ which defines the implicit
equation of the curve $f(x,y)=0$ passing through these 22 points.
%From our previous results, we know that
%$$ \mbox{\rm degree} ~{(f)} \leq 6 . $$
We have $N=22$ and
$$
{{5+2}\choose 2}=21<22,\quad
{{6+2}\choose 2}=28>22. 
$$
Thus, the smallest degree of a polynomial that passes through the 22 points at most $6$,
$$ 
\mbox{\rm degree} ~{(f)} \leq 6 . 
$$
Therefore, it is evident that the line
$y=1$ passes through the points $P_1,~P_7$ and $P_{8}$ but the other $19$ points do not lie on this line, so according to Corollary~\ref{cor-bezout}, this implies that:
$$
\mbox{\rm degree}~({f})\geq \max\{1+1,\lceil\frac{3}{1}\rceil\}=3.
$$
So, we first search for a polynomial of degree $d=3$ in the following form
$$
f(x,y)=c_0+c_1x+c_2y+c_3x^2+c_4xy+c_5y^2+c_6x^3+c_7x^2y+c_8xy^2+c_9y^3.
$$

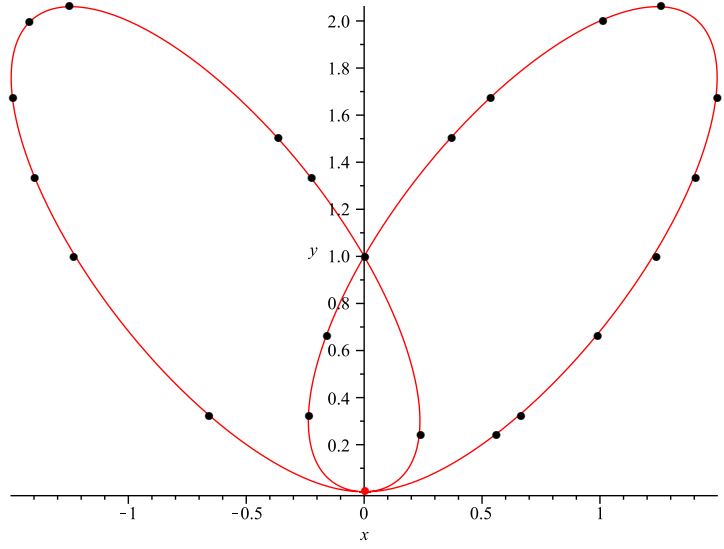
\begin{figure}
\centering
\includegraphics[width=10.00cm]{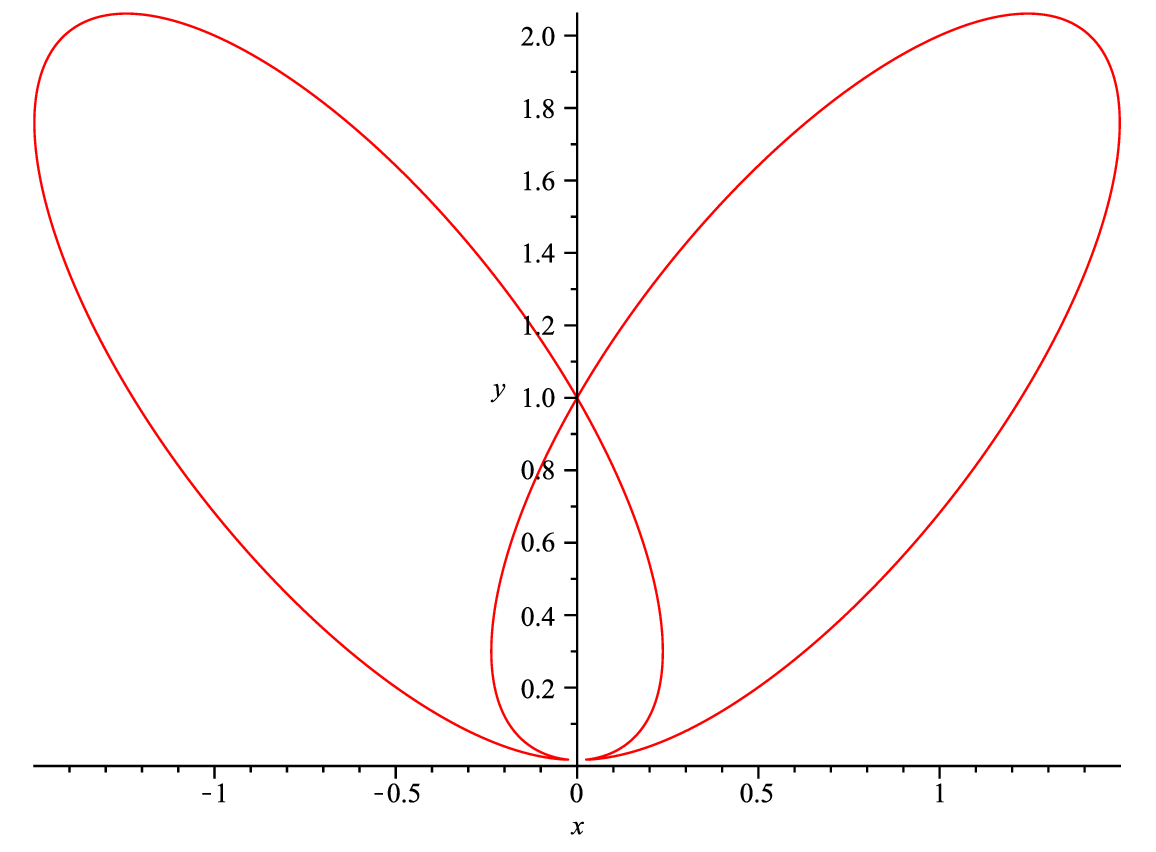}
\begin{picture}(300,20)(-153.5,-42.5)

\put(0,90){\circle*{3}}
\put(21.1,22.5){\circle*{3}}
\put(49.6,22.5){\circle*{3}}

\put(-59,30){\circle*{3}}
\put(59,30){\circle*{3}}
\put(-21.2,30){\circle*{3}}

\put(-14.3,60){\circle*{3}}
\put(88.0,60){\circle*{3}}
\put(-133.2,150){\circle*{3}}
\put(47.4,150){\circle*{3}}
\put(133.2,150){\circle*{3}}

\put(90,179.3){\circle*{3}}
\put(-127.0,179.0){\circle*{3}}
\put(-111.9,185.0){\circle*{3}}
\put(111.9,185.0){\circle*{3}}

\put(-110.227, 90){\circle*{3}}
\put(110.227,90){\circle*{3}}

\put(-20.258, 120.){\circle*{3}}
\put(125., 120.){\circle*{3}}
\put(-125., 120.){\circle*{3}}

\put(-32.80, 135.){\circle*{3}}
\put(32.80, 135.){\circle*{3}}

{\color{red}
\put(-2,1){\line(1,0){4}}
\put(0,1.4142){\circle*{3}}
}
\end{picture}
\caption{The 22 points (small black circles) on the plane selected from the quartic curve
${y}^{4}-2\,{y}^{3}+{y}^{2}-3\,{x}^{2}y+2\,{x}^{4}=0$ for recovering the polynomial.
}
\label{17points}
\end{figure}

Since $(d+1)(d+2)/2=10$, we first choose 10 points from the 22 points, for instance, $P_1,\cdots,P_{10}$. \label{10points}. We then check whether there are cubic curves passing through these 10 points. For this purpose, we substitute the coordinates $x=x_i,y=y_i$ of each point into a $10\times 10$ symmetric matrix generated by
$$[1,x,y,x^2,xy,y^2,x^3,x^2y,xy^2,y^3]^T [1,x,y,x^2,xy,y^2,x^3,x^2y,xy^2,y^3],$$
as follows:
$$
A_i=
 \left( \begin {array}{cccccccccc} 1&x_{i}&y_{i}&{x_{i}^2}&x_{i}y_{i}&{y_{i}^2}&{x_{i}^3}&{x_{i}
^2}y_{i}&x_{i}{y_{i}^2}&{y_{i}^3}\\ \noalign{\medskip}x_{i}&{x_{i}^2}&x_{i}y_{i}&{x_{i}^3}&{x_{i}^
2}y_{i}&x_{i}{y_{i}^2}&{x_{i}^4}&{x_{i}^3}y_{i}&{x_{i}^2}{y_{i}^2}&{y_{i}^3}x_{i}
\\ \noalign{\medskip}y_{i}&x_{i}y_{i}&{y_{i}^2}&{x_{i}^2}y_{i}&x_{i}{y_{i}^2}&{y_{i}^3}&{x_{i}^3}y_{i}&{
x_{i}^2}{y_{i}^2}&{y_{i}^3}x_{i}&{y_{i}^4}\\ \noalign{\medskip}{x_{i}^2}&{x_{i}^3}&{x_{i}}
^{2}y_{i}&{x^4}&{x_{i}^3}y_{i}&{x_{i}^2}{y_{i}^2}&{x_{i}^5}&{x_{i}^4}y_{i}&{x_{i}^3}{y_{i}^2}&
{y_{i}^3}{x_{i}^2}\\ \noalign{\medskip}x_{i}y_{i}&{x_{i}^2}y_{i}&x_{i}{y_{i}^2}&{x_{i}^3}y_{i}&{x_{i}^
2}{y_{i}^2}&{y_{i}^3}x_{i}&{x_{i}^4}y_{i}&{x_{i}^3}{y_{i}^2}&{y_{i}^3}{x_{i}^2}&{y_{i}^4}x_{i}
\\ \noalign{\medskip}{y_{i}^2}&x_{i}{y_{i}^2}&{y_{i}^3}&{x_{i}^2}{y_{i}^2}&{y_{i}^3}x_{i}&
{y_{i}^4}&{x_{i}^3}{y_{i}^2}&{y_{i}^3}{x_{i}^2}&{y_{i}^4}x_{i}&{y_{i}^5}
\\ \noalign{\medskip}{x_{i}^3}&{x_{i}^4}&{x_{i}^3}y_{i}&{x_{i}^5}&{x_{i}^4}y_{i}&{x_{i}^3}
{y_{i}^2}&{x_{i}^6}&{x_{i}^5}y_{i}&{x_{i}^4}{y_{i}^2}&{y_{i}^3}{x_{i}^3}
\\ \noalign{\medskip}{x_{i}^2}y_{i}&{x_{i}^3}y_{i}&{x_{i}^2}{y_{i}^2}&{x_{i}^4}y_{i}&{x_{i}^3}
{y_{i}^2}&{y_{i}^3}{x_{i}^2}&{x_{i}^5}y_{i}&{x_{i}^4}{y_{i}^2}&{y_{i}^3}{x_{i}^3}&{y_{i}^4}{
x_{i}^2}\\ \noalign{\medskip}x_{i}{y_{i}^2}&{x_{i}^2}{y_{i}^2}&{y_{i}^3}x_{i}&{x_{i}^3}{y_{i}}
^{2}&{y_{i}^3}{x_{i}^2}&{y_{i}^4}x_{i}&{x_{i}^4}{y_{i}^2}&{y_{i}^3}{x_{i}^3}&{y_{i}^4}{x_{i}}^
{2}&{y_{i}^5}x_{i}\\ \noalign{\medskip}{y_{i}^3}&{y_{i}^3}x_{i}&{y_{i}^4}&{y_{i}^3}{x_{i}^
2}&{y_{i}^4}x_{i}&{y_{i}^5}&{y_{i}^3}{x_{i}^3}&{y_{i}^4}{x_{i}^2}&{y_{i}^5}x_{i}&{y_{i}^6}
\end {array} \right).
\label{a3expand}
$$
By summing up all the matrices $A_i$, we obtain the matrix $A^{[3]}=\sum_{i=1}^{10} A_i $, which is a $10\times 10$ symmetric matrix given by:
$$
A^{[3]}=
 \left( \begin {array}{cc} U&V\\ \noalign{\medskip}V^T&W\end {array}
 \right),
$$
where $U,V,W$ are $5\times 5$ matrices and
$U, W$ are symmetric, as follows:

$$
\begin{array}{l}
U=
 \left( \begin {array}{ccccc}
 \scalebox{0.7071}{$10$}&\scalebox{0.7071}{$-$}\frac{7}{6}\scalebox{0.7071}{$\sqrt{2}+1$}&{\frac{34}{3}}&{\frac{205}{18}}&
 \scalebox{0.7071}{$-$}{\frac{37}{18}}\scalebox{0.7071}{$\sqrt{2}+2$}\\[7pt]
\scalebox{0.7071}{$-$}\frac{7}{6}\scalebox{0.7071}{$\sqrt{2}+1$} &{\frac{205}{18}}&\scalebox{0.7071}{$-$}{\frac{37}{18}}\scalebox{0.7071}{$\sqrt{2}+2$}&
 \scalebox{0.7071}{$-$}{\frac{217}{108}} \scalebox{0.7071}{$\sqrt{2}+1$}&{\frac{301}{18}}\\[7pt]
{\frac{34}{3}} &\scalebox{0.7071}{$-$}{\frac{37}{18}}\scalebox{0.7071}{$\sqrt{2}+2$}&{\frac{152}{9}}&{\frac{301}{18}}&\scalebox{0.7071}{$-$}{\frac{217}{54}}\scalebox{0.7071}{$\sqrt{2}+4$}\\[7pt]
{\frac{205}{18}} &\scalebox{0.7071}{$-$}{\frac{217}{108}} \scalebox{0.7071}{$\sqrt{2}+1$}&{\frac{301}{18}} &{\frac{6407}{324}}&\scalebox{0.7071}{$-$}{\frac{1297}{324}}\scalebox{0.7071}{$\sqrt{2}+2$}\\[7pt]
\scalebox{0.7071}{$-$}{\frac{37}{18}}\scalebox{0.7071}{$\sqrt{2}+2$} &{\frac{301}{18}}&\scalebox{0.7071}{$-$}{\frac{217}{54}}\scalebox{0.7071}{$\sqrt{2}+4$}&\scalebox{0.7071}{$-$}{\frac{1297}{324}}\scalebox{0.7071}{$\sqrt{2}+2$}&{\frac{4447}{162}}
\end {array} \right),
\end{array}$$
\\
$$\begin{array}{l}
V=
 \left( \begin {array}{ccccc}
 {\frac{152}{9}}&\scalebox{0.7071}{$-$}{\frac{217}{108}}\scalebox{0.7071}{$\sqrt{2}+1$}&{\frac{301}{18}}&
 \scalebox{0.7071}{$-$}{\frac{217}{54}}\scalebox{0.7071}{$\sqrt{2}+4$}&{\frac{766}{27}}\\ [7pt]
   \scalebox{0.7071}{$-$}{\frac{217}{54}}\scalebox{0.7071}{$\sqrt{2}+4$}&{\frac{6407}{324}}&
   \scalebox{0.7071}{$-$}{\frac{1297}{324}}\scalebox{0.7071}{$\sqrt{2}+2$}&
   {\frac{4447}{162}}&\scalebox{0.7071}{$-$}{\frac{1297}{162}}\scalebox{0.7071}{$\sqrt{2}+8$}\\ [7pt]
   {\frac{766}{27}}&\scalebox{0.7071}{$-$}{\frac{1297}{324}}\scalebox{0.7071}{$\sqrt{2}+2$}&{\frac{4447}{162}}&
   \scalebox{0.7071}{$-$}{\frac{1297}{162}}\scalebox{0.7071}{$\sqrt{2}+8$}&{\frac{4088}{81}}\\ [7pt]
   {\frac{4447}{162}}&\scalebox{0.7071}{$-$}{\frac{7777}{1944}}\scalebox{0.7071}{$\sqrt{2}+1$}&{\frac{30223}{972}}&
   \scalebox{0.7071}{$-$}{\frac{7777}{972}}\scalebox{0.7071}{$\sqrt{2}+4$}&{\frac{857}{18}}\\ [7pt]
\scalebox{0.7071}{$-$}{\frac{1297}{162}}\scalebox{0.7071}{$\sqrt{2}+8$}&{\frac{30223}{972}}&
\scalebox{0.7071}{$-$}{\frac{7777}{972}}\scalebox{0.7071}{$\sqrt{2}+4$}&{\frac{857}{18}}&\scalebox{0.7071}{$-$}{\frac{7777}{486}}\scalebox{0.7071}{$\sqrt{2}+16$}\end {array} \right),
\end{array}$$
\\
$$\begin{array}{l}
W=
 \left( \begin {array}{ccccc} 
 {\frac{4088}{81}}&\scalebox{0.7071}{$-$}{\frac{7777}{972}}\scalebox{0.7071}{$\sqrt{2}+4$}&{\frac{857}{18}}&
 \scalebox{0.7071}{$-$}{\frac{7777}{486}}\scalebox{0.7071}{$\sqrt{2}+16$}&{\frac{22534}{243}}\\[7pt] 
\scalebox{0.7071}{$-$}{\frac{7777}{972}}\scalebox{0.7071}{$\sqrt{2}+4$} &{\frac{220879}{5832}}&\scalebox{0.7071}{$-$}{\frac{46657}{5832}}\scalebox{0.7071}{$\sqrt{2}+2$}&{\frac{151571}{2916}}&
 \scalebox{0.7071}{$-$}{\frac{46657}{2916}}\scalebox{0.7071}{$\sqrt{2}+8$}\\ [7pt] 
{\frac{857}{18}} &\scalebox{0.7071}{$-$}{\frac{46657}{5832}}\scalebox{0.7071}{$\sqrt{2}+2$}&{\frac{151571}{2916}}&\scalebox{0.7071}{$-$}{\frac{46657}{2916}}\scalebox{0.7071}{$\sqrt{2}+8$}&{\frac{124375}{1458}}\\[7pt]
\scalebox{0.7071}{$-$}{\frac{7777}{486}}\scalebox{0.7071}{$\sqrt{2}+16$} &{\frac{151571}{2916}}&\scalebox{0.7071}{$-$}{\frac{46657}{2916}}\scalebox{0.7071}{$\sqrt{2}+8$}& {\frac{124375}{1458}}&\scalebox{0.7071}{$-$}{\frac{46657}{1458}}\scalebox{0.7071}{$\sqrt{2}+32$}\\[7pt]
{\frac{22534}{243}} &\scalebox{0.7071}{$-$}{\frac{46657}{2916}}\scalebox{0.7071}{$\sqrt{2}+8$}&{\frac{124375}{1458}}& \scalebox{0.7071}{$-$}{\frac{46657}{1458}}\scalebox{0.7071}{$\sqrt{2}+32$} &{\frac{126752}{729}}
\end {array} \right).
\end{array}
$$

\vspace{0.5cm}
By calling the computer algebra software Maple, we can verify that:
$$
\det(A^{[3]})
=
{\frac {341196800000}{282429536481}}+{\frac {280985600000}{
847288609443}}\,\sqrt {2}
\approx 1.67707205.
$$
According to Theorem~\ref{main-theorem}, we can conclude that no cubic curve $f(x,y)$ passes through the selected 10 points. and thus, it cannot pass through the given 22 points.

Next, we consider the case $d=4$. Since ${4+2\choose 2}=15$ we select 5 additional interpolating points
from the remaining points, for example, $P_{11},\cdots,P_{15}$.
%We then construct the corresponding $15 \times 15$ matrix: $A^{[4]}:=\sum_{i=1}^{15}A_{i}.$
Similarly to before, we calculate the matrix $A_i\,(1\leq i\leq 15)$ for each of the 15 points and then sum them to obtain the $15 \times 15$ matrix $A^{[4]}:=\sum_{i=1}^{15}A_{i}$. Then, we calculate the determinant of $A^{[4]}$ using Maple, and the result is $\det(A^{[4]})=0$. According to Theorem \ref{main-theorem}, this shows that there is a quartic curve passing through $P_1,P_2,\cdots,P_{15}$

To determine the coefficients of the quartic curve, we solve the system of equations:
$$
A^{[4]} \, [c_0,c_1,\cdots,c_{14}]^{T}=[0,0,\cdots,0]^{T}$$
where $[c_0, c_1, \ldots, c_{14}]^T$ represents the coefficient vector. The non-zero coefficients are found to be:
$$
c_5=1,\quad c_7=-3,\quad c_9=-2,\quad c_{10}=2, \quad c_{14}=1.
$$
Hence, $P_1,P_2,\cdots,P_{15}$ lie on the following quartic curve:
$$
f(x,y)=y^2-3x^2y-2y^3+2x^4+y^4.
$$
By substituting the coordinates of the other 7 points into this polynomial, we observe that all the 22 points lie on the quadratic curve.  Finally, it can be asserted that the lowest degree of polynomial curves passing through 22 points simultaneously is 4.
\end{example}

\vspace{0.5cm}

If the data contains errors, we can consider that the position of each point $(x_i, y_i)$ lies within a hemisphere defined by $(x-x_i)^2 + (y-y_i)^2 \leq \delta^2$, or within a rectangle defined by $[x_i-\delta, x_i+\delta] \times [y_i-\delta, y_i+\delta]$. In this case, implicit function interpolation can be formulated as the following optimization problem:
\begin{eqnarray}\label{min-zk} 
   \left\{\begin{array}{ccl}
      &{\displaystyle \min_{\cc_{d}, \, x_i,\,y_i}}  & 
        {\displaystyle \sum_{i=1}^{N}} (f(x_i,y_i,\cc_{d}))^2 \\
       & s.t. &   (x_i,y_i)\in B_i, \quad i=1,2,\cdots,N, \\
        &     & \|\cc_{d}\|_{2}^2=1.
   \end{array}\right.
\end{eqnarray}

In this case, if we find a real-coefficient polynomial of the desired degree, we can perturb it to obtain a rational-coefficient polynomial that passes through the given neighborhood $B_i$. In some practical problems, the given data $(x_i,y_i)$ may have random errors that follow a certain distribution (such as a normal distribution). The goal is to construct a polynomial curve $f(x,y)=0$ that approximately satisfies the equation.
$$
|f(x_i,y_i)|<\epsilon, \quad i=1,2,\cdots,N,
$$
where $\epsilon$ is a small positive number. The optimization problem corresponding to this type of problem is equivalent to adding $N$ linear inequalities of the following form to the constraints of problem \eqref{min-zk}:
$$
-\epsilon<{\cc_{d}^{}}^{T}\,\vv_{d}
=
c_0^{}+c_1^{} \cdot x+c_2^{}\cdot y+ c_{p-2}^{} \cdot x y^{d-1}+ +c_{p-1}^{} \cdot y^d<\epsilon.
$$

Now consider the case in problem \eqref{min-zk} where $B_i$ represents a rectangular neighborhood. By applying the Lagrange multiplier method, the condition for the optimum can be transformed into the following semi-algebraic decision problem:
\begin{eqnarray}
&&\left(\exists (x_1,y_1)\in B_1,\cdots, (x_N,y_N)\in B_N\right)
\wedge\, \exists \cc_d \in \RR^{p} \,
\left(\|\cc_d\|_2=1 \right),\nonumber \\
&&\left(A(x_1,y_1,\cdots,x_N,y_N)\right)_{p \times p} \cc_{d} =[0,0,\cdots,0]^T.
\label{semi-algebra-decision}
\end{eqnarray}

For this problem, applying Theorem \ref{main-theorem} yields the result of the lowest degree of polynomials that pass through all neighborhoods.
%Based on these considerations, we establish a result that provides a lower bound on the degree of the polynomial that passes through the given neighborhoods.
\begin{corollary}\quad
 For $N$ rectangles on the plane $B_i=(\underline{a_i},\overline{a_i})\times (\underline{b_i},\overline{b_i})\,(i=1,2,\cdots,N)$ and a positive integer $d$,
let
$$
V(B_1,B_2,\cdots,B_N)=\{f\in \mathbb{Z}[x,y]\, |\, \mbox{Zero}(f)\cap B_i\not=\emptyset\},
$$
where
$$
\mbox{Zero}(f)=\{(x,y)\in \mathbb{R}^2| f(x,y)=0\},
$$
and
$D_{d}: \mathbb{R}^{2N}\rightarrow R$ a polynomial function defined by
$$
D_{d}(x_1,y_1,\cdots,x_N,y_N)=\det(A^{[d]}(x_1,y_1,\cdots,x_N,y_N)).
$$
If $D_{d}$ is positive definite on the box $B_1\times B_2\times \cdots \times B_N\subset \mathbb{R}^{2N}$ then
\begin{equation}
\min ~ \{\mbox{\rm degree} ~(f) \, |\, ~f\in V(B_1,B_2,\cdots,B_N)\} ~\geq ~d+1.
\label{mindegf}
\end{equation}
\label{cor-mindegf}
\end{corollary}
%%%%%%%%%%%%%%%%%%%%%%%%%%%%%%%%%%%%%%%%%%%%%%%%%%%%%%%%%%%%%%
\begin{proof}\quad
 Since for any $\cc_{d} \in \RR^p$, %where $p={{k+2}\choose 2}-1$, 
we have
\begin{eqnarray*}
&&
\cc_d  \cdot A_d(x_1,y_1,\cdots,x_N,y_N) \cdot {\cc_d}^{T} \\
&=&\sum_{i=1}^N\left(\cc_d \cdot [1,x_i,y_i,\cdots, x_i^d,\cdots,y_i^d]^T(1,x_i,y_i,\cdots,x_i^d,\cdots, y_i^d) \cdot {\cc_d}^{T}\right),\\
&=&\sum_{i=1}^N([1,x_i,y_i,\cdots,x_i^d,\cdots,y_i^d] \cdot \cc_d)^2 ~~ \geq ~~ 0,
%&\geq& 0,
\end{eqnarray*}
the matrix $A^{[d]}$ is a $p \times p $ semi-definite matrix
and $D_{d}\geq 0$ for all $(x_1,y_1,\cdots,x_N,y_N)\in \mathbb{R}$.
Thus, if (\ref{mindegf}) is not valid, then there exists $[c_0,c_1,\cdots,c_{p-1}] 
\in \mathbb{R}^p \setminus\{[0,0,\cdots,0]^{T} \}$ such that
$$
f=c_0+c_1x+c_2y+\cdots+c_{p-1-d}x^d+\cdots+c_{p-1} y^d\in V(B_1,B_2,\cdots,B_N),
$$
that is,
$$
f(x_i,y_i)=c_0+c_1x_i+c_2y_i+\cdots+c_{p-1-d}x_i^d+\cdots+c_{p-1} y_i^d = 0, \quad i=1,2,\cdots, N,
$$
and therefore, for some $(x_i,y_i)\in B_i\ (i=1,2,\cdots,N)$, we have
$$
D_{d}(x_1,x_2,\cdots,x_N,y_N)=0.
$$
\end{proof}
%%%%%%%%%%%%%%%%%%%%%%%%%%%%%%%%%%%%%%%%%%%%%%%%%%%%%%%%%%%%%%%%%%%%%%%

According to Corollary \ref{akeq0}, when $N < {{d+2}\choose 2}$, for any $(x_1, y_1, \cdots, x_N, y_N) \in B_1 \times B_2 \times \cdots \times B_N$, we always have:
$$
D_{d}(x_1,y_1,\cdots,x_N,y_N)=\det(A^{[d]}(x_1,y_1,\cdots,x_N,y_N))= 0,
$$
which means that $D_d$ cannot be positive definite over $B_1 \times B_2 \times \cdots \times B_N$.
When $N \geq {{d+2}\choose 2}$, the determination of whether $D_d$ is positive definite over $B_1 \times B_2 \times \cdots \times B_N$, as stated in Corollary \ref{cor-mindegf}, can be transformed into the following optimization problem:
\begin{equation}
\min\{D_{d}(x_1,y_1,\cdots,x_N,y_N)\, |\,
\underline{a}_i\leq x_i\leq \overline{a}_i,
\underline{b}_i\leq y_i\leq \overline{b}_i,
i=1,2,\cdots,N \}.
\label{minr2n}
\end{equation}
Alternatively, it can be expressed as the following quantifier elimination problem:
\begin{equation}
\begin{array}{ll}
\exists (x_1,y_1,\cdots,x_N,y_N) \nonumber \\
\phantom{\exists} (\underline{a}_i\leq x_i\leq \overline{a}_i,
\underline{b}_i\leq y_i\leq \overline{b}_i, i=1,\cdots,N)
\wedge D_{d}(x_1,y_1,\cdots,x_N,y_N)=0.
\end{array}
\label{qe-r2n}
\end{equation}

Note that the number of terms in $D_{2}$ increases from $125865$ for $N=6$ to $782775$ for $N=7$.
To simplify $D_{d}$, we introduce the matrix $M_d$ of order $p \times N$ defined as follows.
%The following result is used to reduce the complexity of $D_{d}$ .
%%%%%%%%%%%%%%%%%%%%%%%%%%%%%%%%%%%%%%%%%%%%%%%%%%%%%%%%%%%%%%%%%
\begin{definition}
Let $d$ be a positive integer, $p = {d+2 \choose 2}$, and $(x_1, y_1), (x_2, y_2), \cdots, (x_N, y_N)$ be $N$ points in the plane. 
The $p \times N$ matrix $M_d$ defined below is referred to as the $d$th order implicit function interpolation basis matrix for these $N$ points:
$$
M_d = M(x_1,y_1,x_2,y_2,\cdots,x_{N},y_{N}):=
 \left(
 \begin {array}{llcr}
 1&1&\cdots&1\\
 x_1&x_2&\cdots&x_N\\
 y_1&y_2&\cdots&y_N\\
 \vdots&\vdots&\ddots&\vdots\\
 x_1^d&x_2^d&\cdots&x_N^d\\
 x_1^{d-1}y_1&x_2^{d-1}y_2&\cdots&x_N^{d-1}y_N\\
 \vdots&\vdots&\ddots&\vdots\\
 y_1^d&y_2^d&\cdots&y_N^d
 \end {array}
 \right).
$$
\end{definition}
Given $N$ points, there is a connection between the $d$-th order implicit function interpolation basis matrix $M_d$ and the matrix $A^{[d]}$.
\begin{lemma}\quad
%Let $k>0$ be an integer, $p={{k+2}\choose 2}-1$, and 
%$B_k^T$ is the transpose of matrix $B_k$.
For any $N$ points $\{(x_i,y_i), i= 1 \cdots N \}\in \mathbb{R}^{2N}$, we have:
$$
A^{[d]}
=M_d \cdot M_d^T.
$$
n particular, when $N = p$, we also have:
$$
\det (A^{[d]})= \det (M_d ^2).
$$
\label{akbktbk}
\end{lemma}
\begin{proof}\quad

$$ M_d \cdot M_d^T= \left(
 \begin {array}{llcr}
 1&1&\cdots&1\\
 x_1&x_2&\cdots&x_N\\
 y_1&y_2&\cdots&y_N\\
 \vdots&\vdots&\ddots&\vdots\\
 x_1^d&x_2^d&\cdots&x_N^d\\
 x_1^{d-1}y_1&x_2^{d-1}y_2&\cdots&x_N^{d-1}y_N\\
 \vdots&\vdots&\ddots&\vdots\\
 y_1^d&y_2^d&\cdots&y_N^d
 \end {array}
 \right).
\left(
 \begin {array}{cccccccc}
 1&   x_1  & y_1& \cdots& x_1^d&   x_1^{d-1}y_1& \cdots & y_1^d   \\
 1&   x_2  & y_2& \cdots& x_2^d&   x_2^{d-1}y_2& \cdots & y_2^d \\
 \vdots & \vdots &\ddots& \vdots &\vdots & \ddots & \vdots \\
 1&   x_N  & y_N& \cdots& x_N^d&   x_N^{d-1}y_N& \cdots & y_N^d
 \end {array}
 \right),
%\hfill
$$

\begin{equation}
= \scalebox{0.9}{$
\left(\begin{array}{cccccccc}
N             &\left[x_i\right]      &\left[y_i\right]           &\cdots
&\left[x_i^d\right]        &\left[x_i^{d-1}y_i\right]      &\cdots&\left[y_i^d\right]         \\[5pt]

\left[x_i\right]         &\left[x_i\, x_i\right]   &\left[x_i\, y_i\right]        &\cdots
&\left[x_i\, x_i^{d}\right]    &\left[x_i\, x_i^{d-1}y_i\right]   &\cdots&\left[x_i\, y_i^d\right]      \\ [5pt]

\left[y_i\right]         &\left[y_i\, x_i\right]   &\left[y_i\, y_i\right]         &\cdots
&\left[y_i\, x_i^d\right]     &\left[y_i\, x_i^{d-1}y_i\right]      &\cdots&\left[y_i\, y_i^{d}\right] \\[5pt]

\vdots        &\vdots     &\vdots          &\ddots&\vdots         &\vdots              &\ddots&\vdots           \\[5pt]

\left[x_i^d\right]       &\left[x_i^{d}\, x_i\right]&\left[x_i^d\, y_i\right]      &\cdots
&\left[x_i^{d}\, x_i^d\right]     &\left[x_i^d\, x_i^{d-1}y_i\right]     &\cdots&\left[x_i^d\, y_i^d\right]      \\[5pt]

\left[x_i^{d-1}y_i\right]&\left[x_i^{d-1}y_i\, x_i\right] &\left[x_i^{d-1}y_i\, y_i\right]&\cdots
&\left[x_i^{d-1}y_i\, x_i^d\right]&\left[x_i^{d-1}y_i\, x_i^{d-1}y_i\right]   &\cdots&\left[x_i^{d-1}y_i\, y_i^d\right]\\[5pt]

\vdots        &\vdots     &\vdots          &\ddots&\cdots         &\vdots              &\ddots&\vdots              \\ [5pt]

\left[y_i^d\right]       &\left[y_i^d\, x_i \right] &\left[y_i^{d}\, y_i\right]     &\cdots&
\left[y_i^d\, x_i^d\right]   &\left[y_i^d\, x_1^{d-1}y_i\right]&\cdots&\left[y_1^{d}\, y_i^d\right]
\end{array}\right).
$} \\[7pt]
\label{bbt}
\end{equation}

\noindent Here, we use the notation $[x_i], [y_i], [x_iy_i], \ldots, [x_i^d], [x_i^{d-1}y_i],\ldots, [y_i^d]$ to record
\begin{align*}
[x_i]&=x_1+x_2+\cdots+x_N, \\
[y_i]&=y_1+y_2+\cdots+y_N, \\
[x_iy_i]&=x_1y_1+x_2y_2+\cdots+x_Ny_N,\\
&\cdots\\
[x_i^d]&=x_1^d+x_2^d+\cdots+x_N^d, \\
[x_i^{d-1}y_i]&=x_1^{d-1}y_1+x_2^{d-1}y_2+\cdots+x_N^{d-1}y_N,\\
&\cdots\\
[y_i^d]&=y_1^d+y_2^d+\cdots+y_N^d.
\end{align*}

Comparing these definitions with the definition of the matrix $A^{[d]}$ in Definition \ref{def:1}, we can observe that the matrix (\ref{bbt}) is exactly equal to the expanded form of $A^{[d]}$.
%(See page \pageref{a3expand} for $A_3$ as an example).
\end{proof}
\vspace{0.5cm}
\noindent The corollary below follows directly from Corollary~\ref{cor-mindegf}.
\begin{corollary}\quad
Consider rectangles $B_i=(\underline{a_i},\overline{a_i})\times (\underline{b_i},\overline{b_i})$ for $i=1,2,\cdots,N)$,  
%$k$ a positive integer, $p+1=(k+1)(k+2)/2$, and
and $\Delta_d$ a polynomial function defined on $\mathbb{R}^{2p}$ by
$$
\Delta_d=\det(M_d(x_1,y_1,\cdots,x_{p},y_{p})).
$$
If
$\Delta_d$ is positive definite or negative definite on the set $
B_{j_1}\times B_{j_2}\times \cdots \times B_{j_p}\subset \mathbb{R}^{2p}$
for at least one $ (j_1,j_2,\cdots,j_p)$, with $1\leq j_1<j_2<\cdots <j_p\leq 2N$, then
$$
\min ~ \{\mbox{\rm degree}(f) \, |\, f\in V(B_1,B_2,\cdots,B_N)\}\geq d+1.
$$
\label{cor-psd}
\end{corollary}

The degree of the polynomial $\Delta_d$ is given by
$$
\mbox{\rm degree}(\Delta_d)=\sum_{j=1}^d j(j+1)=\frac{1}{3}\,d \left( d+1 \right)  \left( d+2 \right).
$$
The degrees with respect to $x_i$ and $y_i$ are
$$
\quad \mbox{\rm degree}(\Delta_d,x_i)=d, \quad \mbox{\rm degree}(\Delta_d, y_i)=d, \text{for all} i=1,2,\cdots,p.$$

The expanded form of $\Delta_d$ contains $p!$ monomials. For example, when $d=2$, the minimum value of $p$ is ${4+2\choose 2}=6$. Therefore, the total degree of the polynomial $\Delta_2$ is 8, and the number of monomials it contains is at least $p!$, which is 720.

%In the following example, with $d=2$ and $p=6$, $\Delta_2$ has 12 variables. We use this specific example to illustrate how to approximate the positivity of an $8$th-degree polynomial with 12 variables in a small rectangular neighborhood in $R^{12}$. Consequently, we can determine the minimal degree required for polynomial implicit function interpolation in this case.

%%%%%%%%%%%%%%%%%%%%%%%%%%%%%%%%%%%%%%%%%%%%%%
%%%%%%%%%%%%%%%%%%%%%%%%%%%%%%%%%%%%%%%%%%%%%%%%%%%%%%
\begin{example} \label{ex2}\;
The following data represents the approximate values of the average distance from the Sun (semi-major axis, which is the mean of perihelion and aphelion distances) and the orbital period of the six closest planets to the Sun in the solar system: Mercury, Venus, Earth, Mars, Jupiter, and Saturn. Each data point $(r,p)$ is represented by a rectangle on a plane, with the side length of the square being $0.0006$. Here is an example for the data of Mercury: $\mbox{Mer}:=[0.3868, 0.3874]\times[0.2405, 0,2411]$. The unit of $r$ is in astronomical units ($A.U.$), where
$1A.U.=149,597,870$ kilometers. The unit of $p$ is the Earth's orbital period, which is used as a reference for time.

\begin{equation}
\left.\begin{array}{cll}
&\mbox{Mercury):}&\mbox{\phantom{xx} Mer}=[0.3868, 0.3874]\times[0.2405, 0,2411],\\
&\mbox{Venus):}&\mbox{\phantom{xx} Ven}= [0.7230, 0.7236]\times[0.6149, 0.6155],\\
&\mbox{Earth:}&\mbox{\phantom{xx} Erd}=[1, 1]\times[1, 1],\\
&\mbox{Mars:}&\mbox{\phantom{xx} Mar}=[1.5234, 1.5240]\times[1.8805, 1.8811], \\
&\mbox{Jupiter:}&\mbox{\phantom{xx} Jup}=[5.2041, 5.2047]\times[11.8617, 11.8623], \\
&\mbox{Saturn:}&\mbox{\phantom{xx} Sat}=[9.5823, 9.5829]\times[29.4568, 29.4574].
\end{array}
\right\}
\label{no-quadratic}
\end{equation}
Determine whether there exists a quadratic curve on the plane passing through the six given rectangles.

First, we take the centers of the six rectangles as $(x_i^{\ast},y_i^{\ast})$, and represent each rectangle using $x_i=x_i^{\ast}+\Delta x_i$ and $y_i=y_i^{\ast}+\Delta y_i$.
\begin{eqnarray*}
&\mbox{Mer}:& x_1=\frac{3871}{10000}+\Delta x_1,  \quad  y_1=\frac{2408}{10000}+\Delta y_1, \\[10pt]
&\mbox{Ven}:& x_2=\frac{7233}{10000}+\Delta x_2,\quad  y_2=\frac{6152}{10000}+\Delta y_2, \\[10pt]
&\mbox{Erd}:& x_3=1,\quad y_3=1, \\[10pt]
&\mbox{Mar}:& x_4=\frac{15237}{10000}+\Delta x_4, \quad   y_4=\frac{18808}{10000}+\Delta y_4,\\[10pt]
&\mbox{Jup}:& x_5=\frac{52044}{10000}+\Delta x_5, \quad  y_5=\frac{11862}{10000}+\Delta y_5,\\[10pt]
&\mbox{Sat}:& x_6=\frac{95826}{10000}+\Delta x_6, \quad  y_6=\frac{294571}{10000}+\Delta y_6.
\end{eqnarray*}
Secondly, substitute $x_i,y_i$ to calculate the determinant below
$$
M_2=
 \left| \begin {array}{cccccc} 1&{x_1}&{y_1}&{{x_1}}^{2}&{x_1}\,{y_1}&{{y_1}}^{2}\\ \noalign{\medskip}1&{x_2}&{y_2}&{{x_2}}^{2}&{x_2}\,{y_2}&{{y_2}}^{2}\\
 \noalign{\medskip}1&{x_3}&{y_3}&{{x_3}}^{2}&{x_3}\,{y_3}&{{y_3}}^{2}\\
 \noalign{\medskip}1&{x_4}&{y_4}&{{x_4}}^{2}&{x_4}\,{y_4}&{{y_4}}^{2}\\
 \noalign{\medskip}1&{x_5}&{y_5}&{{x_5}}^{2}&{x_5}\,{y_5}&{{y_5}}^{2}\\
 \noalign{\medskip}1&{x_6}&{y_6}&{{x_6}}^{2}&{x_6}\,{y_6}&{{y_6}}^{2}
 \end {array} \right|.
$$
Then $M_2(\Delta x_1,\Delta y_1, \cdots, \Delta x_6,\Delta y_6)$ is a polynomial of degree 8 with 720 monomials:
$$
M_2={\frac{728327184304772538780278933931}{781250000000000000000000000000}}
+\sum_{j=1}^{8}H_k(\Delta x_1,\Delta y_1, \ldots, \Delta x_6,\Delta y_6),
$$
Here, $H_k(\Delta x_1,\Delta y_1, \ldots, \Delta x_6,\Delta y_6),(k=1,2,\ldots,8)$ represents a $k$th-degree homogeneous polynomial. For example:
\begin{align*}
H_1&=c_{11}\Delta x_1+c_{12}\Delta y_1+\ldots+c_{112}\Delta y_6,\;&\mbox{ ($12$ monomials)}\\
H_2&=c_{21}\Delta x_1^2+\ldots+c_{26}\Delta x_6^2+c_{27}\Delta y_1^2+\ldots+c_{212}\Delta y_6^2\\
&\phantom{xxxxxx}+c_{213}\Delta x_1y_1+c_{214}\Delta x_1y_2+\ldots+c_{248}\Delta x_6y_6, \;&\mbox{ ($48$ monomials)}
\end{align*}
where $c_{11},\ldots,c_{112}, c_{21},\ldots,c_{248}$ are specific rational numbers. For brevity, their exact values are not shown here.
It is easy to see that for any positive number $\delta > 0$ and variables $\Delta x_1, \Delta y_1, \cdots, \Delta x_5, \Delta y_5 \in [-\delta, \delta]$, the following inequalities hold:
\begin{align*}
H_1&\geq -(|c_{11}|+|c_{12}|+\ldots+|c_{112}|)\delta,\;\\
H_2&=-(|c_{21}|+\ldots+|c_{27}|+\ldots+|c_{213}|+\ldots+|c_{248}|)\delta^2.
\end{align*}
Similarly, for $H_3,\ldots,H_8$, when $\Delta x_1, \Delta y_1, \cdots, \Delta x_5, \Delta y_5 \in [-\delta, \delta]$, the following inequalities hold:
$$
H_k\geq -C_k\times \delta^k, \quad 3\leq k\leq 8,
$$
where $C_k$ is a positive number determined by the coefficients of the polynomial $H_k$. By substituting specific numerical values, we obtain the inequality:
\begin{eqnarray*}
M_2&\geq&
{\frac{728327184304772538780278933931}{781250000000000000000000000000}}
-{\frac {15486384648337660843877367483}{6250000000000000000000000}}\,{\delta}\\[10pt]
&&
-{\frac {4025428097433897230769192269}{125000000000000000000000}}\,{\delta}^{2}
-{\frac {2430088906604087412321393}{12500000000000000000}}\,{\delta}^{3}\\[10pt]
&&
-{\frac {264230612082325102439}{625000000000000}}\,{\delta}^{4}
-{\frac {155328292156034629}{500000000000}}\,{\delta}^{5}\\[10pt]
&&
-{\frac {2943856054691}{50000000}}\,\delta^6
-{\frac {11662287}{2500}}\,{\delta}^{7}
-120\,{\delta}^{8}:=k_2(\delta)
\label{b2-delta}
\end{eqnarray*}
%holds for all $\delta>0$ and $\Delta x_1, \Delta y_1, \cdots, \Delta x_5, \Delta y_5 \in [-\delta, \delta]$.
It is easy to see that $k_2(0)=0.932258\cdots>0$, the smallest positive real root of $k_2(\delta)=0$ is
$0.000374\cdots$, and the largest negative real root is $-1.024959\cdots$.
This implies that when $-1.024960<\delta<0.000374$ 
$$M_2>k_2(\delta)>0$$
which means that for $(x_i,y_i)$, $1\leq i\leq 6$, within the following rectangle:
\begin{equation}
[x_i^{\ast}-\frac{36}{100000},x_i^{\ast}+\frac{36}{100000}]
\times
[y_i^{\ast}-\frac{36}{100000},y_i^{\ast}+\frac{36}{100000}]
\label{6boxes}
\end{equation}
where $(x_i^{\ast},y_i^{\ast})$ represents the center of each of the six small rectangles, the polynomial $M_2(x_1,y_1,\ldots,x_6,y_6)$ is positive definite.

Finally, it is verified that the square regions $(r,p)$ of the six planets given in \eqref{no-quadratic} (with side length $0.0003$) are all contained within the six rectangular regions described in \eqref{6boxes} (with side length $0.00036$ and identical centers). In other words, no quadratic curves in terms of $r$ and $p$ pass through the six square regions in \eqref{no-quadratic}.
\end{example}

\section{The sparsity aspect of approximate implicit interpolation}
\label{sparsity}
%In this section, we delve into the sparsity aspect of the approximate implicit interpolation, focusing on its application in two distinct scenarios: real data representation and theoretical problems. 

Consider a set of $N$ points $P_i=(x_i,y_i)$ on the two-dimensional plane $\mathbb{R}^2$  and a given positive integer $d$.
In the previous section, we learned that if $d>\lfloor \sqrt{2N}\rfloor$, there always exists an algebraic curve of degree $d$ passing through the given points. However, if $d<\lfloor \sqrt{2N}\rfloor$, there may not be such a curve passing through all of these points simultaneously.
In this section, we discuss how to find a polynomial curve $f(x,y)=0$ of degree at most $d$ such that the difference between these points and the curve, represented by $\sum_{i=1}^N f(x_i,y_i)^2$, is minimized.

To approach this problem, we assume that $f(x,y)$ IS expressed as:
$$
f(x,y)=c_0+c_1\, x+c_2\, y+\cdots+c_{p-d}\, x^k+\cdots+c_{p-1}\, y^d,
$$
where $c_0,c_1,\cdots, c_{p-1}$ are real numbers satisfying the condition:
$$ c_0^2+c_1^2+\cdots+ c_{p-1}^2=1.$$
To solve this optimization problem, we employ the Lagrange multiplier method. The problem can be formulated as follows:
\begin{eqnarray}\label{min-sparsity}
   \left\{\begin{array}{ccl}
      &{\displaystyle \min }  & 
        {\displaystyle g(c_0,c_1,\ldots,c_{p-1}):=\sum_{i=1}^{N}} (c_0+c_1x_i+c_2y_i+\cdots+c_{p-1}y_i^d)^2 \\
       & s.t. & c_0^2+c_1^2+\cdots+c_{p-1}^2=1. 
   \end{array}\right.
\end{eqnarray}
The Lagrangian function for this problem is defined as:
$$
\mathcal{L}:=\sum_{i=1}^N \left(c_0+c_1x_i+c_2y_i+\cdots+c_{p-1}y_i^d\right)^2
+\lambda \cdot (c_0^2+c_1^2+\cdots+ c_{p-1}).
$$
By calculating the partial derivatives of $\mathcal{L}$ with respect to $c_0,c_1,\cdots,c_{p-1}$, we obtain the first-order conditions for the optimization problem \eqref{min-sparsity}:
$$
\frac{\partial \mathcal{L}}{\partial c_0}=0, \quad
\frac{\partial \mathcal{L}}{\partial c_1}=0, \quad
\frac{\partial \mathcal{L}}{\partial c_2}=0, \quad
\cdots,
\quad
\frac{\partial \mathcal{L}}{\partial c_{p-1}}=0.
$$
These equations can be expressed in matrix form as:
\begin{equation}
\left(
\frac{\partial^2 \mathcal{L}}{\partial c_i\partial c_j}
\right)_{i,j=0,1,\cdots,p-1}
\left(
\begin{array}{l}
c_0\\c_1\\\vdots\\c_{p-1}
\end{array}
\right)
=
\left(
\begin{array}{c}
0\\0\\\vdots\\0
\end{array}
\right),
\label{eq-homo}
\end{equation}
where
$$
\frac{\partial^2 \mathcal{L}}{\partial c_i \partial c_j}\,= 
\left\{\begin{array}{ll}
\dfrac{\partial^2 g}{\partial c_i^2}+2 \lambda,&\mbox{if $j=i$},\\[7pt]
\dfrac{\partial^2 g}{\partial c_i \partial c_j}, &\mbox{if $j\not=i$},
\end{array}
\right.
\quad (i,j=0,1,\cdots,p-1).
$$
In this way, by substituting the right-hand sides of the two equations into the equation
 $$\det
\left(
\frac{\partial^2 \mathcal{L}}{\partial c_i\partial c_j}
\right) = 0$$
we obtain a polynomial equation in $\lambda$:
\begin{equation}
a_0+a_1\lambda+a_2\lambda^2+\cdots+a_{p-1}\lambda^{p-1}+2^{p}\lambda^{p}=0,
\label{charmx}
\end{equation}
which is actually the characteristic equation of the matrix
$$
M(x_1,y_1,\ldots,x_N,y_N)=\left(-\frac{1}{2}\frac{\partial^2 g}{\partial c_i\partial c_j}\right)
$$
Since the matrix $M$ is a real symmetric matrix, the equation \eqref{charmx} has at least one real root. After solving for $\lambda$, we substitute it into the homogeneous linear equation \eqref{eq-homo} to solve for the desired coefficients $c_0,c_1,\ldots,c_{p-1}$.

when $N<p-1$, the coefficients $a_0,\cdots,a_{p-1-N}$in \eqref{charmx} are zero,
resulting in the solution to \eqref{min-sparsity} forming a linear subspace of dimension $p-N$ in $\mathbb{R}^{p}$. The intersection of this linear space with the unit sphere, defined as:
$$
S^{p-1}:=\{(c_0,c_1,\cdots,c_{p-1})\,|\, c_0^2+c_1^2+\cdots+c_{p-1}^2=1\}
$$
is an infinite set.
Therefore, further exploration is required to find the sparse polynomial $f(x,y)$ that passes through the given points.
For instance, when considering the 10 points $P_1,P_2,\cdots, P_{10}$ as provided in Example~\ref{ex1}, any quartic curve 
$$
f(x,y)=c_0+c_1x+c_2y+\cdots+c_{10}x^4+c_{11}x^3y+\cdots+c_{14}y^4
$$
passing through these 10 points satisfies the equation:
\begin{equation}
(c_0,c_1,\cdots,c_9)^T= M_{10\times 5} (c_{10},c_{11},\cdots,c_{14})^T,
\label{c104}
\end{equation}
where
$$
M_{10\times 5}=
 \left( \begin {array}{clcrc}
 \scalebox{0.8}{$0$}&-{\frac{95}{47}}-{\frac {20\,\sqrt {2}
}{47}}&-{\frac{56}{47}}+{\frac {8\,\sqrt {2}}{47}}&{\frac{40}{141}}+{
\frac {16\,\sqrt {2}}{47}}&\scalebox{0.8}{$0$}\\
\noalign{\medskip}\scalebox{0.8}{$0$}&-{\frac {25\,\sqrt
{2}}{282}}-{\frac{10}{47}}&-{\frac{35}{141}}+{\frac {5\,\sqrt {2}}{141
}}&-{\frac{70}{141}}+{\frac {10\,\sqrt {2}}{141}}&\scalebox{0.8}{$0$}
\\
\noalign{\medskip}\scalebox{0.8}{$0$}&-{\frac {15\,\sqrt {2}}{94}}-{\frac{242}{141}}&
-{\frac{21}{47}}+{\frac {3\,\sqrt {2}}{47}}&-{\frac{136}{47}}+{\frac {
6\,\sqrt {2}}{47}}&\scalebox{0.8}{$0$}\\
\noalign{\medskip}\scalebox{0.8}{$0$}&-{\frac {119\,\sqrt {2}}{
282}}-{\frac{10}{47}}&{\frac{59}{141}}-{\frac {14\,\sqrt {2}}{47}}&{
\frac {10\,\sqrt {2}}{141}}+{\frac{8}{47}}&\scalebox{0.8}{$0$}\\
\noalign{\medskip}\scalebox{0.8}{$0$}&{
\frac {125\,\sqrt {2}}{141}}+{\frac{488}{141}}&{\frac{350}{141}}-{
\frac {50\,\sqrt {2}}{141}}&{\frac{277}{141}}-{\frac {100\,\sqrt {2}}{
141}}&\scalebox{0.8}{$0$}\\
\noalign{\medskip}-\frac{3}{2}&{\frac {119\,\sqrt {2}}{282}}+{\frac{
10}{47}}&-{\frac{200}{141}}+{\frac {14\,\sqrt {2}}{47}}&-{\frac {10\,
\sqrt {2}}{141}}-{\frac{8}{47}}&\scalebox{0.8}{$0$}\\
\noalign{\medskip}-{\frac{34}{9}}&
-{\frac {427\,\sqrt {2}}{423}}-{\frac{445}{141}}&-{\frac {623\,\sqrt {
2}}{141}}-{\frac{1252}{423}}&-{\frac {824\,\sqrt {2}}{423}}+{\frac{356
}{141}}&{\frac{77}{9}}\\
\noalign{\medskip}-\frac{5}{9}&-{\frac {35\,\sqrt {2}
}{423}}-{\frac{25}{47}}&-{\frac{380}{423}}-{\frac {35\,\sqrt {2}}{47}}
&-{\frac {160\,\sqrt {2}}{423}}+{\frac{20}{47}}&{\frac{10}{9}}
\\
\noalign{\medskip}\frac{3}{2}&{\frac {7\,\sqrt {2}}{47}}+{\frac{45}{47}}&{
\frac{40}{141}}+{\frac {63\,\sqrt {2}}{47}}&{\frac {32\,\sqrt {2}}{47}
}-{\frac{36}{47}}&-\scalebox{0.8}{$5$}\\
\noalign{\medskip}{\frac{17}{6}}&{\frac {133\,
\sqrt {2}}{141}}+{\frac{385}{141}}&{\frac{168}{47}}+{\frac {539\,
\sqrt {2}}{141}}&-{\frac{308}{141}}+{\frac {232\,\sqrt {2}}{141}}&-{
\frac{17}{3}}\end {array} \right).
$$
And therefore, the search of the sparse quartic polynomial $f(x,y)=0$
that passes the 10 points leads to an optimization problem as follows:
\begin{eqnarray}
&\min &||(c_0,c_1,c_2,\cdots,c_{14})||_0, \label{zero-norm-14}\\
&\mbox{\rm s.t.}& (c_0,c_1,\cdots,c_9)^T= M_{10\times 5} (c_{10},c_{11},\cdots,c_{14})^T, \nonumber\\
&&c_0^2+c_1^2+\cdots+c_{14}^2=1, \nonumber
\end{eqnarray}
where the zero-norm $||\cdot||_0$ represents the zero-norm of a vector (i.e., the number of its non-zero elements). This optimization problem aims to find a sparse quartic polynomial by minimizing the number of non-zero coefficients.
%With the symbolic computation software {\sc Maple} we can verify the following two facts

To solve this problem, we can employ symbolic computation software such as Maple. By using Maple, we can verify the following facts:

%%%%%%%%%%%%%%%%%%%%%%%%%%%%%%%%%%%%%%%%%%%%%%%%%%%%%%%%%%%%%%%%%%%%%%%
\begin{itemize}
  \item For any selection $j_1,j_2,\cdots,j_{11}$ with $0\leq j_1<j_2<\cdots<j_{11}\leq 14$ the equation system formed by (\ref{c104}) has a non-zero solution satisfying
  $$
  c_{j_1}=c_{j_2}=\cdots=c_{j_{11}}=0
  $$
There are ${15 \choose 11}=1365$ such selections.
  
  \item Among all ${15 \choose 10}=3003$ subsets $S\subset \{0,1,2,\cdots,14\}$ with $\#(S)=10$, there exist exactly two subsets $S_1,S_2$ that satisfy (\ref{c104}) and $\{c_j=0\,|\, j\in S\}$ with non-zero solutions $N\!Z$. They are:
  $$
  S_1=\{0, 1, 2, 3, 4, 6, 8, 11, 12, 13\},
  $$
  $$
  N\!Z_1=\{
  c_5 = \frac{1}{\sqrt{19}}, c_7 = -\frac{3}{\sqrt{19}}, c_9 = -\frac{2}{\sqrt{19}},
  c_{10} = \frac{2}{\sqrt{19}}, c_{14} = \frac{1}{\sqrt{19}}
  \}
  $$

  $$
  S_2=\{1, 3, 4, 6, 7, 8, 10, 11, 12, 13\},
  $$
  $$
  N\!Z_2=\{
  c_0 = \frac{5/2}{\sqrt{671}},
  c_2 = -\frac{51/4}{\sqrt{671}},
  c_5 = \frac{77/4}{\sqrt{671}},
  c_9 = -\frac{45/4}{\sqrt{671}},
  c_{14} = \frac{9/4}{\sqrt{671}}
  \}.
  $$
  The corresponding curves are given by the equations
  $$
  {y}^{2}-3\,{x}^{2}y-2\,{y}^{3}+2\,{x}^{4}+{y}^{4}=0,
  $$
  and
  $$
  \frac{5}{2}-{\frac {51}{4}}\,y+{\frac {77}{4}}\,{y}^{2}-{\frac {45}{4}}\,{y}^{3}+\frac{9}{4}\,{y}^{4}=0.
  $$
\end{itemize}

The previous quartic curve is consistent with the curve obtained by implicit function interpolation using 15 points in Example \ref{ex1}. The latter curve is actually composed of four lines:
$$
  (y-1)(y-2)(3y-1)(3y-5)=0.
  $$
%%%%%%%%%%%%%%%%%%%%%%%%%%%%%%%%%%%%%%%%%%%%%%%%%%%%%%%%%%%%%%%%%%%%%%%
When the degree $d$ is large, solving the optimization problem \eqref{c104} directly through exhaustive search becomes computationally challenging. In such cases, we can approximate the solution by solving the following convex optimization problem:
\begin{eqnarray}
&\min& \sum_{0 \leq i<j\leq 14} c_i^2c_j^2= \frac{1}{2}~(-c_0^4-c_1^4-\cdots-c_{14}^4),\label{convex}\\
&\mbox{\rm s.t.}& (c_0,\cdots,c_9)^T=M_{10\times 5}(c_{10},\cdots,c_{14})^T, \nonumber\\
&&c_0^2+c_1^2+\cdots+c_{14}^2=1. \nonumber
\end{eqnarray}

By formulating the problem in this way, we transform it into a convex optimization problem that can be efficiently solved using available optimization algorithms. The resulting solution provides an approximate sparse quartic polynomial that satisfies the given conditions. For some practical applications, obtaining an approximate solution with a certain level of accuracy within the allowed time is sufficient. For theoretical problems, an approximate solution can also serve as a foundation for further research.

To illustrate our approach and demonstrate its practicality and effectiveness in real-world scenarios, we further investigate the analysis of planetary motion data in Example \ref{ex1}.
%\subsection{Example of real data of planet motion}
\begin{example} \textbf{Analysis of the actual motion data of planets and asteroids in the solar system.}

Consider the data of the eight major planets, 23 minor planets, and 7 dwarf planets orbiting the solar system, including the semi-major axis lengths of their orbits and the orbital periods (data source: https://en.wikipedia.org/wiki/Planet). In the previous example (Example \ref{ex2}) of the previous section, data for the first 6 planets, which are closer to Earth, were provided. The data for Uranus, Neptune, minor planets, and dwarf planets are given in Table \ref{tab2}. The observational errors in the data, denoted by $r$, are within $\pm 0.001 AU$, and the errors in the orbital periods, denoted by $p$, are within $\pm 0.01$ years.

\begin{table}[!h]
\caption{Orbit data coordinates for Uranus, Neptune, 23 asteroids, and 7 dwarf planets.}{}
\label{tab2}
%%%%%%%%%%%%%%%%%%%%%%%%%%%%%%%%%%%%%%%%%%%%%%%%%%%%%%%%%%%%%%%%%%%%%%%
\begin{center}
\begin{tabular}{lrr||lrr} \hline
asteroid&$r$ (in {\it Au\/})&$p$ (in {\it yr\/})&asteroid&$r$ (in {\it Au\/})&$p$ (in {\it yr\/})\\ \hline
Uranus\/{\rm ($\ast\ast$)}&19.2184&84.0205&
Neptune\/{\rm ($\ast\ast$)}&30.1104&164.7693\\ \hline
Ceres\/{\rm ($\ast$)}&2.768&4.60&
Pallas&2.772&4.61\\ \hline
Juno&2.671&4.36&
Vesta&2.362&3.63\\ \hline
Astraea&1.574&1.97&
Hebe&2.426&3.78\\ \hline
Iris&2.385&3.68&
Flora&2.202&3.27\\ \hline
Metis&2.387&3.69&
Hygiea&3.1421&5.57\\ \hline
Parthenope&2.453&3.84&
Victoria&2.334&3.57\\ \hline
Egeria& 2.576&4.13&
Irene&2.150&3.16\\ \hline
Eunomia&2.643&4.30&
Psyche&2.921&4.99\\ \hline
Thetis&2.470&3.88&
Melpomene&2.296&3.48\\ \hline
Fortune&2.441&3.81&
Proserpina&2.656&4.33\\ \hline
Bellona&2.777&4.63&
Amphitrite&2.554&4.08\\ \hline
Leukothea&2.990&5.17&
Fides&2.641&4.29\\ \hline
Chiron\/{\rm ($\ast$)}&13.708&50.76&
Pluto\/{\rm ($\ast$)}&39.48&247.94\\ \hline
Haumea\/{\rm ($\ast$)}&43.218&284.12&
Makemake\/{\rm ($\ast$)}&45.715&309.09\\ \hline
Eris\/{\rm ($\ast$)}&67.781&558.04&
Sedna \/{\rm ($\ast$)}&506.8&11411 \\ \hline
\end{tabular}
\end{center}
\vspace{-2pt}

{\small
\hspace{50pt}
In the table, the symbol $*$ represents a dwarf planet, and the symbol $**$ represents a planet.\hfill}
%%%%%%%%%%%%%%%%%%%%%%%%%%%%%%%%%%%%%%%%%%%%%%%%%%%%%%%%%%%%%%%%%%%%%%%
\end{table}

Figure \ref{keplerlaw} presents the orbital data for planets in the solar system and some selected minor and dwarf planets shown in Table \ref{tab2}. The theoretical dimensions of each rectangle in the figure are $0.02\times 0.002$. For clarity, the actual display uses the {\tiny $\blacksquare$} symbol from the tiny size package in LaTeX, which is horizontally and vertically magnified by a factor of $2$ and $20$ respectively.

From the previous section's Example \ref{ex2}, we know that no quadratic curve passes through these rectangular neighborhoods simultaneously. Our goal is to determine whether there exist cubic curves that pass through the $6$ small rectangular neighborhoods mentioned in \eqref{no-quadratic}, as well as $32$ small rectangular neighborhoods centered around the data in Table \ref{tab2} with dimensions $0.002\times 0.01$.

If such cubic curves exist, we aim to find a curve with the fewest possible terms. The ideal scenario would involve re-discovering Kepler's Third Law using rigorous analytical methods based on this actual data.

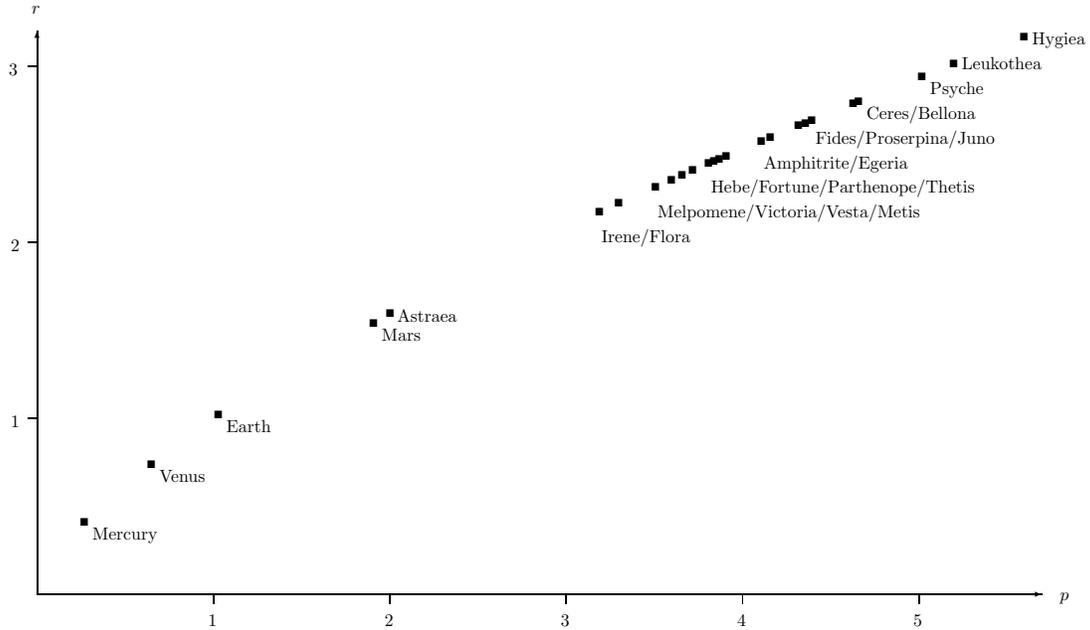
\begin{figure}
\centering
\scalebox{0.667}
{
\begin{picture}(600,330)(-20,-10)
\put(0,0){\vector(1,0){570}}
\put(0,0){\vector(0,1){320}}
\put(580,-3){$p$}
\put(-3,330){$r$}
\put(100,0){\line(0,-1){5}}
\put(97,-18){$1$}
\put(200,0){\line(0,-1){5}}
\put(197,-18){$2$}
\put(300,0){\line(0,-1){5}}
\put(297,-18){$3$}
\put(400,0){\line(0,-1){5}}
\put(397,-18){$4$}
\put(500,0){\line(0,-1){5}}
\put(497,-18){$5$}
\put(0,100){\line(-1,0){5}}
\put(-15,95){$1$}
\put(0,200){\line(-1,0){5}}
\put(-15,195){$2$}
\put(0,300){\line(-1,0){5}}
\put(-16,295){$3$}
\put(24,39){{\tiny $\blacksquare$}} \put(28,31){ Mercury}
\put(62,72){{\tiny $\blacksquare$}} \put(66,64){ Venus}
\put(100,100){{\tiny $\blacksquare$}} \put(104,92){ Earth}
\put(188,152){{\tiny $\blacksquare$}} \put(192,144){ Mars}
\put(197.44,157.38){{\tiny $\blacksquare$}} \put(201,155){ Astraea}
\put(316,215){{\tiny $\blacksquare$}} \put(320,200){Irene/Flora}
\put(327,220.2){{\tiny $\blacksquare$}}
\put(348,229.6){{\tiny $\blacksquare$}} \put(352,214){Melpomene/Victoria/Vesta/Metis}
\put(357,233.4){{\tiny $\blacksquare$}} %%%Victoria
\put(363,236.179){{\tiny $\blacksquare$}} %%%Vesta
%%%\put(368,238.5){{\tiny $\blacksquare$}} %%%Iris
\put(369,238.7){{\tiny $\blacksquare$}} %%%Metis
\put(378,242.6){{\tiny $\blacksquare$}}  \put(382,228){Hebe/Fortune/Parthenope/Thetis}
\put(381,244.1){{\tiny $\blacksquare$}} %%% Fortune
\put(384,245.3){{\tiny $\blacksquare$}} %%% Parthenope
\put(388,247){{\tiny $\blacksquare$}} %%% Thetis
\put(408,255.4){{\tiny $\blacksquare$}} \put(412,242){Amphitrite/Egeria}
\put(413,257.6){{\tiny $\blacksquare$}}
\put(429,264.1){{\tiny $\blacksquare$}}  \put(438,256){ Fides/Proserpina/Juno}
%%%\put(430,264.3){{\tiny $\blacksquare$}} %%%{Eunomia}
\put(433,265.6){{\tiny $\blacksquare$}}
\put(436.463,267.07){{\tiny $\blacksquare$}}
\put(460,276.75){{\tiny $\blacksquare$}} \put(467,270){ Ceres/Bellona}
\put(463,277.7){{\tiny $\blacksquare$}}
\put(499,292.1){{\tiny $\blacksquare$}}  \put(503,284){ Psyche}
\put(517,299){{\tiny $\blacksquare$}} \put(521,298){ Leukothea}
\put(557,314.21){{\tiny $\blacksquare$}} \put(561,312){ Hygiea}
\end{picture}
}
\caption{The orbital periods of some planets and asteroids in the solar system and the semi-major axes of their orbits.
The size of {\tiny $\blacksquare$} is adjusted to reflect the real error of the data.
}
\label{keplerlaw}
\end{figure}

To achieve this goal, our focus is on constructing cubic curves of the form:
\begin{equation}
k(r,p)=c_0+c_1r+c_2p+c_3r^2+c_4rp+c_5p^2+c_6r^3+c_7r^2p+c_8rp^2+c_9p^3=0,
\label{krp}
\end{equation}
such that they pass through the point $(1,1)$ and the 37 other small square regions. We aim to maximize the number of coefficients $c_0, c_1, \ldots, c_9 \in \mathbb{R}$ that are zero, while ensuring that at least one of $c_6, c_7, c_8, c_9$ is non-zero.

Without loss of generality, we assume that the first non-zero coefficient $c_i$ is positive, and the sum of squares of all coefficients is equal to 1, i.e.,
$$
c_0^2+c_1^2+\cdots+c_9^2=1. 
$$

To avoid accumulation of errors caused by rounding during calculations, we represent the 38 data points as rational points $P_i=(x_i,y_i)$, where $i=1,2,\ldots,38$, as shown in Table \ref{tab3}.

\begin{table}[!h]
\caption{Rational representation of $38$ data points.}{}
\label{tab3}
%%%%%%%%%%%%%%%%%%%%%%%%%%%%%%%%%%%%%%%%%%%%%%%%%%%%%%%%%%%%%%%%%%%%%%%
\begingroup
\setlength{\tabcolsep}{1pt} % Default value: 6pt
\renewcommand{\arraystretch}{1.3} % Default value: 1
$$\begin{array}{|c|c|c|c|c|c|c|c|c|c|}
\hline
\phantom{jj}P_1\phantom{hh} &\phantom{jj}P_2\phantom{hh}&\phantom{jj}P_3\phantom{hh} &\phantom{jj} P_4\phantom{hh} &\phantom{jj} P_5\phantom{hh} 
&\phantom{jj} P_6\phantom{hh} &\phantom{jj} P_7\phantom{hh} &\phantom{jj} P_8\phantom{hh} &\phantom{jj}  P_9\phantom{hh} 
&\phantom{j}  P_{10}\phantom{hh}  \\ \hline
 \frac{3871}{10000} & \frac{7233}{10000} & {\small 1} & \frac{15237}{10000} & \frac{13011}{2500} & \frac{47913}{5000} & \frac{24023}{1250} & \frac{3011}{100} & \frac{987}{25} & \frac{346}{125}   \\
\hline
 \frac{301}{1250} & \frac{769}{1250} & {\small 1} & \frac{2351}{1250} &\frac{5931}{500} & \frac{294571}{10000} & \frac{168041}{2000} &\frac{1647693}{10000} & \frac{12397}{50}& \frac{23}{5} \\
\hline
 P_{11} &  P_{12} &  P_{13} &
 P_{14}& P_{15} & P_{16} &  P_{17} &  P_{18} &  P_{19} &  P_{20}  \\
\hline
 \frac{693}{250} & \frac{2671}{1000} & \frac{1181}{500}& \frac{787}{500}& \frac{1213}{500} & \frac{477}{200} & \frac{1101}{500} & \frac{2387}{1000} &\frac{31421}{10000} & \frac{2453}{1000} \\
 \hline
 \frac{461}{100} & \frac{109}{25} & \frac{363}{100}  & \frac{197}{100}&\frac{189}{50}& \frac{92}{25}& \frac{327}{100}& \frac{369}{100}&\frac{557}{100}& \frac{96}{25}\\
\hline
P_{21} &  P_{22} &  P_{23} &  P_{24} &  P_{25} & P_{26} & P_{27} & P_{28} & P_{29} &P_{30}  \\
\hline
 \frac{1167}{500} & \frac{322}{125} & \frac{43}{20} & \frac{2643}{1000} &  \frac{2921}{1000} & \frac{247}{100} & \frac{287}{125} & \frac{2441}{1000} & \frac{332}{125} & \frac{2777}{1000} \\
\hline
\frac{357}{100}& \frac{413}{100}&\frac{79}{25}& \frac{43}{10} & \frac{499}{100} & \frac{97}{25} &\frac{87}{25}&\frac{381}{100}& \frac{433}{100}&\frac{463}{100}\\ \hline
P_{31} &P_{32} &P_{33} &P_{34} &P_{35} &P_{36} &P_{37}&P_{38}& &  \\\hline
\frac{1277}{500} &\frac{299}{100} & \frac{2641}{1000}& \frac{3427}{250} & \frac{67781}{1000}& \frac{2534}{5} &\frac{21609}{500} & \frac{9143}{200}& & \\ \hline
\frac{102}{25}& \frac{517}{100} & \frac{429}{100}&\frac{1269}{25}&\frac{13951}{25}&{\small 11411}& \frac{7103}{25}&\frac{30909}{100}& & \\ \hline
\end{array}$$
 \label{38points}
 \endgroup
%%%%%%%%%%%%%%%%%%%%%%%%%%%%%%%%%%%%%%%%%%%%%%%%%%%%%%%%%%%%%%%%%%%%%%%
\end{table}
 
It can be easily seen that if any four points $[x_1,y_1],\ldots,[x_4,y_4]$ from the set of 38 data points lie on a cubic curve given by $f(x, y) = c_0 + c_1x + c_2y + \cdots + c_6x^3 + \cdots + c_9y^3$, then the coefficients $c_0, c_1, \ldots, c_9$ satisfy the following system of linear equations: 
 
\begin{equation}
\left(\begin{array}{cccccccccc}
1&x_1&y_1&x_1^2&x_1y_1&y_1^2&x_1^3&x_1^2y_1&x_1y_1^2&y_1^3\\
1&x_2&y_2&x_2^2&x_2y_2&y_2^2&x_2^3&x_2^2y_2&x_2y_2^2&y_2^3\\
1&x_3&y_3&x_3^2&x_3y_3&y_3^2&x_3^3&x_3^2y_3&x_3y_3^2&y_3^3\\
1&x_4&y_4&x_4^2&x_4y_4&y_4^2&x_4^3&x_4^2y_4&x_4y_4^2&y_4^3
\end{array}\right)
\left(\begin{array}{c}
c_0\\
c_1\\
\vdots\\
c_9
\end{array}\right)
=
\left(\begin{array}{c}
0\\
0\\
\vdots\\
0
\end{array}\right).
\label{cubic4}
\end{equation}

First, we construct nine quadruples $G_1, G_2, \ldots, G_9$ using the 38 data points. Among them, $G_1$ and $G_2$ contain the eight major planets of the solar system, $G_2$ to $G_8$ include 23 minor planets and the dwarf planet Ceres (the largest known minor planet in the asteroid belt). $G_9$ consists of four dwarf planets.

In Table~\ref{tab3}, two dwarf planets (Eris and Sedna) are not included in the grouping and are reserved for validation purposes. The specific construction of the nine quadruples is as follows.

\begin{eqnarray*}
\mbox{\it Giant planets}&& G_1:=\{\mbox{\rm Mercury, Venus, Earth, Mars}\}; \\
&& G_2:=\{\mbox{\rm Jupiter, Saturn, Uranus, Neptune}\}; \\
\mbox{\it Asteroids }&& G_3:=\{\mbox{\rm Pallas, Juno. Vesta, Astraea}\};\\
&& G_4:=\{\mbox{\rm Hebe, Iris, Flora, Metis}\}; \\
&& G_5:=\{\mbox{\rm Hygiea, Parthenope, Victoria, Egeria}\};\\
&& G_6:=\{\mbox{\rm Irene, Eunomia, Psyche, Thetis}\};\\
%\end{eqnarray*}
%\begin{eqnarray*}
&& G_7:=\{\mbox{\rm Melpomene, Fortune, Proserpina, Bellona}\};\\
&& G_8:=\{\mbox{\rm Amphitrite, Leukothea, Fides, Ceres}\};\\
\mbox{\it Dwarf planets}&& G_9:=\{\mbox{\rm Chiron, Pluto, Haumea, Makemake}\}.
\end{eqnarray*}
Next, for each group of four stars, we perform the following calculations:
Construct a set $S_{(10,5)}$ of five integers ${j_1,j_2,\cdots,j_5}$ satisfying $0\leq j_1<j_2<\cdots<j_5\leq 5$. 

Consider the set $J={j_1,j_2,\cdots,j_5}$ consisting of five distinct integers such that $0\leq j_1<j_2<\cdots<j_5\leq 9$, and the set ${6,7,8,9}$ is not a subset of $J$. We construct the set $S_{(10,5)}$ of all such combinations. The number of elements in $S_{(10,5)}$ is 246. For each such set of five integers (Integer Quintet) $J$, we define:
$$
\{j_6,j_7,\cdots,j_{10}\}=\{0,1,2,\cdots, 9\}\setminus J, ~~ j_6<j_7<\cdots<j_{10}.
$$
Then by substituting
$
c_{j_1}=c_{j_2}=c_{j_3}=c_{j_4}=c_{j_5}=0
$
into the equation system (\ref{cubic4}) we get four homogeneous equations with $5$ unknowns $j_6,j_7,\cdots,j_{10}$.
%%%%%%%%%%%%%%%%%%%%%%%%%%%%%%%%%%%%%%%%%%%%%%%%%%%%%%%%%%%%%%%%%%%%%%%
For each non-zero solution $c_{j_6},c_{j_7},\ldots,c_{j_{10}}$ of this derived system of equations, we obtain a non-zero solution $\phi(J):=(c_0,c_1,\cdots,c_9)$ for \eqref{cubic4}, where at least five of the coefficients are zero. For example, for the four-star group $G_1={\text{Mercury, Venus, Earth, Mars}}$ and the next two sets of five integers from $S_{(10,5)}$:
$$J_1=\{0,1,2,3,7\}, ~ J_2=\{[0, 3, 5, 6, 9]\}$$ 
the solutions to the system of equations \eqref{cubic4} are respectively:
$$
\phi(J_1)=
\scalebox{0.866}{
$(0, 0, 0, 0, -0.005115\cdots, -.701381\cdots, .712633\cdots, 0, -0.012353\cdots, 0.006218\cdots),$
}
$$
$$
\phi(J_2)=
\scalebox{0.866}{
$(0, 0.154246\cdots, -0.407426\cdots, 0, 0.582026\cdots, 0, 0, -0.621253\cdots, 0.292406\cdots, 0)$.
}
$$

It took only $0.514$ seconds 
%in the computer mentioned in page~\pageref{machine-type-page}
to generate $\phi(J)$ for all $J\in S_{(10,5)}$.

For each $\phi(J)=(c_0,c_1,\cdots,c_9),$ we define
$$
\psi(J)=
\{i ~|~0\leq i\leq 9 ~\text{and}~ |c_i|<\frac{1}{100\sqrt{10}}\}.
$$
Here, we take $1/100\sqrt{10}$ as the threshold to set $c_i$ to zero if its absolute value is smaller than this value, based on the following consideration:
When the system of equations (\ref{cubic4}) and $J$ are generated using a random number generator, $(c_0, c_1, \ldots, c_9)$ is uniformly distributed on the 10-dimensional unit sphere, and the mean value of each $|c_i|$ is $1/\sqrt{10}$. The probability that the absolute value is smaller than $1/(100\sqrt{10})$ is approximately given by

$$
p=\frac{S_{10}(-\frac{1}{100\sqrt{10}},-\frac{1}{100\sqrt{10}})}{S_{10}}\approx\frac{\frac{2}{100\sqrt{10}}\times V_9}{S_{10}}=\frac{64}{7875\sqrt{10}\cdot \pi}
\approx
%%%\frac{1}{560.499\cdots}=
0.000818\cdots,
$$ 
where $S_{10}$ is the surface area of the 10-dimensional unit sphere and $V_9$ is the volume of the 9-dimensional unit sphere, given by
$$
S_{10}=\frac{2\pi^{10/2}}{\Gamma(\frac{10}{2})}=\frac{\pi^{5}}{12},\quad
V_9=\frac{\pi^{9/2}}{\Gamma(\frac{9}{2}+1)}=\frac{32\pi^4}{945},
$$
and $S_{10}(-\delta, \delta)$ is the domain consisting of points on the 10-dimensional unit sphere within a distance of $\delta$ from the equator. When $\delta \approx 0$, we have
$$
S_{10}(-\delta,\delta)\approx 2\delta\times V_{9}.
$$

It is evident that $J \subset \psi(J)$. When $J = \psi(J)$, the solutions of the linear system \eqref{cubic4} are all non-zero elements. When $J \subsetneq \psi(J)$, the solutions of the linear system \eqref{cubic4} contain new zero elements, i.e., $\#\psi(J)\geq 6$. Since we are looking for cubic curves with coefficients $c_0, c_1, \ldots, c_9$ that have as many zero elements as possible, for the set $S_{(10,5)}$ containing 246 elements, we first identify the subset of index sets $J$ in $\mathcal{S}'$ that satisfy the condition $J \subsetneq \psi(J)$, and then compute $\mathcal{S}'' = {\psi(J) | J \in \mathcal{S}'}$ and discuss them accordingly.

For $G_1$, computation shows that $\mathcal{S}'$ contains 57 elements, while $\mathcal{S}''$ contains 18 elements. Based on the number of zero coefficients, they can be further divided into three categories:
\begin{eqnarray}
&S_{(10,6)}=&\left\{\scalebox{0.9}{$
[0, 1, 2, 3, 4, 6], [0, 1, 2, 3, 5, 6], [0, 1, 2, 3, 5, 9], [0, 1, 2, 3, 6, 8], [0, 1, 2, 3, 6, 9]
$}\right.;
\nonumber%%%\label{s106a}
\\%[8pt]
&&\phantom{xxx}\left.\scalebox{0.9}{$
[0, 1, 2, 3, 4, 5], [0, 1, 2, 3, 4, 7], [0, 1, 2, 3, 4, 9], [0, 1, 2, 3, 5, 8]
$};\right.
\nonumber%%%\label{s106}
\\%[8pt]
&&\phantom{xxx}\left.\scalebox{0.9}{$
[0, 1, 2, 3, 8, 9], [0, 1, 2, 4, 8, 9], [0, 1, 2, 7, 8, 9], [0, 1, 3, 7, 8, 9]
$}\right\};
\label{s106c}
\\%[8pt]
&S_{(10,7)}=&\{\scalebox{0.9}{$
[0, 1, 2, 3, 4, 7, 9], [0, 1, 2, 3, 4, 8, 9], [0, 1, 2, 3, 7, 8, 9], [0, 1, 2, 4, 7, 8, 9]
$}\};
\label{s107}
\\%[8pt]
&S_{(10,8)}=&\{\scalebox{0.9}{$
[0, 1, 2, 3, 4, 7, 8, 9]
$}\}.
\label{s108}
\end{eqnarray}

Note that each element $J=[j_1,j_2,\ldots,j_k]$ in $S_{(10,k)},(k=6,7,8)$ contains $k$ numbers from $0,1,\ldots,9$, arranged in ascending order. $S_{(10,6)}$ shows that among all cubic equations \eqref{krp} passing through the four rectangular regions of $G_1={\text{Mercury}, \text{Venus}, \text{Earth}, \text{Mars}}$, there are precisely 6 zero elements in the coefficients. Disregarding the specific numerical values of the non-zero coefficients, these equations must belong to one of the following 13 forms: 
 
\begin{eqnarray}
p^2+r^2p+rp^2+p^3,\quad rp+r^2p+rp^2+p^3,\quad rp+r^3+r^2p+rp^2,\phantom{xxx}\nonumber \\
rp+p^2+r^2p+rp^2,\quad rp+p^2+r^2p+p^3,\phantom{xxx} \nonumber \\
r^3+r^2p+rp^2+p^3,\quad p^2+r^3+rp^2+p^3,\quad p^2+r^3+r^2p+rp^2,\quad rp+r^3+r^2p+p^3,\phantom{xxx}\label{s106b-polynomials} \\
rp+p^2+r^3+r^2p,\quad r^2+p^2+r^3+r^2p,\quad r^2+rp+p^2+r^3,\quad p+rp+p^2+r^3.\phantom{xxx} \label{s106c-polynomials}
\end{eqnarray}

Among these 13 forms, the first 5 forms are reducible and can be factored out with either $r$ or $p$. According to the conclusion from Example \ref{ex1} in the previous section, these forms are actually impossible to occur. $S_{(10,7)}$ shows that among all cubic equations \eqref{krp} passing through $G_1$, there are exactly 7 zero elements in the coefficients, with 4 possible forms:

\begin{equation}
c_5p^2+c_6r^3+c_8rp^2,\quad c_5p^2+c_6r^3+c_7r^2p,\quad c_4rp+c_5p^2+c_6r^3,\quad c_3r^2+c_5p^2+c_6r^3.
\label{s107-polynomials}
\end{equation}

Finally, $S_{(10,8)}$ shows that among all cubic equations \eqref{krp} passing through $G_1$, there are exactly 8 zero elements in the coefficients, and they must have the form:
\begin{equation}
c_5p^2+c_6r^3=0.
\label{s108-polynomials}
\end{equation}

Let $\chi({G_1})$ be the set of all $13$ forms of cubic polynomials included in equations \eqref{s106b-polynomials}, \eqref{s106c-polynomials}, \eqref{s107-polynomials}, and \eqref{s108-polynomials}, namely:

\begin{eqnarray*}
&&\chi({G_1})=\left\{
r^3+r^2p+rp^2+p^3,\quad p^2+r^3+rp^2+p^3,\quad p^2+r^3+r^2p+rp^2,
\right.
\\
&&
\phantom{\chi({G_1})=xxx}rp+r^3+r^2p+p^3, \quad rp+p^2+r^3+r^2p,\quad r^2+p^2+r^3+r^2p,
\\
&&
\phantom{\chi({G_1})=xxx}r^2+rp+p^2+r^3,\quad p+rp+p^2+r^3,\quad p^2+r^3+rp^2,
\\
&&
\phantom{\chi({G_1})=xxx}\left.p^2+r^3+r^2p,\quad rp+p^2+r^3,\quad r^2+p^2+r^3,\quad p^2+r^3\right\}.
\end{eqnarray*}

We can now assert that among the cubic polynomials determined by the four small rectangular regions of $G_1$, if at least $6$ of their coefficients are zero, then these cubic polynomials, disregarding the specific values of the non-zero coefficients, must belong to one of the $13$ forms included in $\chi({G_1})$. Hence, we refer to $\chi({G_1})$ as the set of {\it sparse cubic curve types} supported by the four-star configuration $G_1$.
Currently, we hope that the results obtained from the data through implicit function interval interpolation, which corresponds formally to the polynomial $c_5p^2+c_6r^3=0$, are also included in the set $\chi({G_1})$.

Apply the same method to $G_2, G_3, \ldots, G_9$. We aim to extract the common supported {\it sparse cubic curve types/} from $\chi(G_i)$ for $1\leq i\leq 9$. The results reveal that:

$$
\chi(G_1)=\chi(G_2)=\cdots=\chi(G_9).
$$

Since the construction of $\chi(G_i)$ for a selected four-star configuration $G_i$ can be completed on average in $0.5$ seconds, we employ the Monte Carlo method to randomly generate four-star configurations $G_X$ among the $38$ data points listed in Table \ref{tab3}. Our expectation is to obtain a distinct sparse type $\chi(G_X)$ from $\chi(G_i)$ for $i=1,2,\ldots,9$. After conducting $1000$ experiments, we consistently obtained $\chi(G_X)=\chi(G_1)$. It is worth noting that there are a total of ${38\choose 4}=73,815$ possible four-star configurations that can be constructed from the $38$ data points. While it is feasible to perform a thorough computation by traversing through all possibilities, our objective is to find the simplest form of a cubic curve. Therefore, in this case, we directly compute the best results obtained by using the binomial expression:

$$
k_j(r,p)=c_5p^2+c_6r^3
$$
to construct the implicit functions passing through all $38$ data points $P_i=[x_i,y_i]$ listed on page \pageref{38points}.

Substituting $x_i$ and $y_i$ into the equation, we obtain:
\begin{equation}
\sum_{i=1}^{38}(y_i^2\cdot c_5+ x_i^3 c_6)^2=\alpha c_5^2+\beta c_5c_6+ \gamma c_6^2,
\label{kepler38}
\end{equation}
where
\begin{eqnarray*}
&&\alpha=33910002280102731124343238312183/(2\times 10^{15}),\\
&&\beta=42374075982683010624551482072423759/(125\times 10^{16}),\\
&&\gamma=4236065336850550592183712032020535827017/(25\times 10^{22}).
\end{eqnarray*}
 The values of $c_5,c_6$ that minimize $\alpha c_5^2+\beta c_5c_6+ \gamma c_6^2$ satisfy
\begin{center}
\begin{multline*}
-\frac{c_5}{c_6}=
\scalebox{0.9}{
${\displaystyle \frac{42374075982683010624551482072423759}{42387502850128413905429047890228750}}$
}\\
=
\scalebox{0.9}{
$\cfrac{1}{1+\cfrac{1}{3155+\cfrac{1}{1+\cfrac{1}{11+\cfrac{1}{75+\cdots}}}}}$}
\approx
\scalebox{0.9}{$-0.999683\cdots$},
\end{multline*}
\end{center}
and the approximate fractions are
$$
1, \quad \frac{3155}{3156}, \quad \frac{3156}{3157},
\quad \frac{37871}{37883}, \quad \frac{2843481}{2844382}, \cdots.
$$

Using these values, we obtain the following cubic curves:
$$
k_1(r,p)=p^2-r^3,\quad
k_2(r,p)=p^2-\frac{3155}{3156}r^3,\quad
k_3(r,p)=p^2-\frac{3156}{3157}r^3,\quad
%%%k_4(r,p)=p^2-\frac{37871}{37883}r^3,\quad
\cdots
$$

Given that the errors associated with the data are smaller than $0.001$ AU for $r$ and $0.01$ years for $p$, a box $B_i$ can be assigned to each point $[x_i, y_i]$ as follows:
$$
B_i:=\{(x,y)\,|\,x_i-\frac{1}{1000}\leq x\leq x_i+\frac{1}{1000}, \, y_i-\frac{1}{100}\leq y\leq y_i+\frac{1}{100}\},
\quad i=1,2,\cdots,38.
$$

The following conclusions can be verified:
\begin{itemize}
\item The curve $k_1(r,p)= p^2-r^3$ intersects with each $B_i$, except for $i \neq 5,6,7,8,38$.
\item The curves $k_2(r,p)=0$ and $k_3(r,p)=0$ intersect with each $B_i$, except for $i=5, 6, 7, 8, 35, 36, 37, 38$. Note that $P_5,P_6,P_7,P_8$ correspond to the planets Jupiter, Saturn, Uranus, and Neptune in the solar system, while $P_{35},P_{36},P_{37},P_{38}$ represent the dwarf planets Haumea, Makemake, Eris, and Sedna, respectively.
\end{itemize}

It should be noted that the analysis of orbital data shows that the formulas $k_2(r,p)=0$ and $k_3(r,p)=0$ are more accurate than the Kepler's Third Law $k_1(r,p)=p^2-r^3$. This is because in reality, the precise formula relating $r$ and $p$ is given by:
$$
\frac{p^2}{r^3}=\frac{4\pi^2}{G(M_{\odot}+m)}, 
$$
where $M_{\odot}$ is the mass of the Sun, $m$ is the mass of the planet (asteroid or dwarf planet), and $G$ is the gravitational constant. The equation $p^2-r^3=0$ is only correct when $m<<M_{\odot}$.

\end{example}
\section{The uniqueness problem of algebraic curve interval interpolation with specified monomials and integer coefficients.}

In some proof problems, symbolic computation methods face difficulties in controlling the space complexity during intermediate calculations. In such cases, it is possible to instantiate some variables and use symbolic-numeric hybrid computations to obtain a large number of simplified solutions to the problem. Then, utilizing implicit function interpolation techniques such as algebraic curves or surfaces, the solutions of the instantiated problem can be integrated into solutions of the original problem. Often, through theoretical analysis, we can anticipate that the integrated solutions form a polynomial or surface with integer coefficients.

In this context, we not only aim to obtain sparse solutions with fewer terms, but also desire the coefficients of the resulting polynomials to be reasonably small. Additionally, we would like to determine the maximum range of coefficient values for which the solutions remain unique as integer-coefficient polynomials.

In this section, we discuss the following problem of implicit polynomial interval interpolation: constructing a polynomial curve that passes through a given set of vertical short lines on the plane. We focus on studying the uniqueness of integer-coefficient polynomial interpolation.

We have the following theorem.
\begin{theorem}\quad
Let $m_1(x,y)=x^{k_1}y^{l_1},m_2(x,y)=x^{k_2}y^{l_2},$ and $m_q(x,y)=x^{k_q}y^{l_q}$ be monomials of $x,y$ and
$
f(x,y)=c_1m_1(x,y)+c_2m_2(x,y)+\cdots+c_{q}m_q(x,y)
$
a polynomial that satisfies $|c_j|\leq 1, j=1,2,\cdots,q$ and passes through the line segments
$$
S_i=\{(x,y)|x=x_i>0, 0<y_i-\delta\leq y\leq y_i+\delta\}, i=1,2,\cdots,N.
$$
Then
$$
\sum_{i=1}^N (f(x_i,y_i))^2\leq
\delta^2\,q \cdot \sum_{i=1}^N h(x_i,y_i+\delta),
$$
where
$$
h(x,y)=\sum_{j=1}^{q} \left(\frac{\partial m_j(x,y)}{\partial y}\right)^2.
$$\label{estimation}
\end{theorem}
%%%%%%%%%%%%%%%%%%%%%%%%%%%%%%%%%%%%%%%%%%%%%%%%%%%%%%%%%%%%%%%%%%%%%%%
\begin{proof}\quad
 Assume that $f(x,y)$ passes the following points on each of the given line segments
at point $P_i: (x_i,y_i+\Delta y_i)$, where $|\Delta y_i|\leq \delta$ for $i=1,2,\cdots,N$.
Then, for $i=1,2,\cdots,N$ we have
\begin{eqnarray*}
(f(x_i,y_i))^2&=&(f(x_i,y_i)-f(x_i,y_i+\Delta y_i)^2 \\[10pt]
&=&{\displaystyle \sum_{j=1}^{q}\left(c_{j}(x_i^{k_j}\, y_i^{l_j}-x_i^{k_j}\,(y_i+\Delta y_i)^{l_j})\right)^2}\\[10pt]
&\leq&{\displaystyle (c_1^2+c_2^2+\cdots+c_{q}^2)\cdot
\sum_{j=1}^{q} x_i^{2k_j}\left(y_i^{l_j}-(y_i+\Delta y_i)^{l_j}\right)^2}\\[10pt]
&=&{\displaystyle q\cdot \sum_{j=1}^{q} x_i^{2k_j}
\left(\Delta y_i\cdot l_j\cdot (y_i+\theta_{l_j}\Delta y_i)^{l_j-1}\right)^2\,\quad (0< \theta_{l_j}< 1)}\\[10pt]
%\end{eqnarray*}
%\begin{eqnarray*}
&\leq&{\displaystyle q\cdot \sum_{j=1}^{q} x_i^{2k_j}((l_j)\delta\cdot (y_i+\delta)^{l_j-1})^2}\\[10pt]
&=&{\displaystyle \delta^2\,q\cdot \sum_{j=1}^{q}l_j^2\,x_i^{2k_j}\left(y_i+\delta\right)^{2l_j-2}
=\delta^2\,q\cdot \sum_{j=1}^{q}\left(\frac{\partial m_j}{\partial y}(x_i,y_i+\delta)\right)^2},
\end{eqnarray*}
\end{proof}
\begin{corollary}\quad
Let $p=(d+1)(d+2)/2$. If a polynomial
$$
f(x,y)=c_0+c_1x^2+c_2y^2+\cdots+c_{p-d}x^{2d}+c_{p-d+1}x^{2d-2}y^2+\cdots+c_{p-1}y^{2d}, $$
with $|c_j|\leq 1$ for $j=0,1,\cdots, p$
%%%$$
%%%c_0^2+c_1^2+\cdots+c_{p-1}^2=1,
%%%$$
passes through the line segments
$$
S_i=\{(x,y)|x=x_i>0, 0<y_i-\delta\leq y\leq y_i+\delta\}, i=1,2,\cdots,N.
$$
Then
$$
\sum_{i=1}^N (f(x_i,y_i))^2\leq
2(d+1)(d+2)\delta^2\cdot \sum_{i=1}^N h_d(x_i,y_i+\delta),
$$
where
$$
h_d(x,y)=y^2\sum_{l,j\geq 0,l+j\leq d-1}^{} (l+1)^2x^{4j}y^{4l}.
$$
In particular,
\begin{eqnarray*}
&& h_2(x,y)={y}^{2} \left( 1+{x}^{4}+4\,{y}^{4} \right),
~~ h_3(x,y)={y}^{2} \left( 1+{x}^{4}+4\,{y}^{4}+{x}^{8}+4\,{x}^{4}{y}^{4}+9\,{y}^{8} \right),\\
&& h_4(x,y)={y}^{2} \left( 1+{x}^{4}+4\,{y}^{4}+{x}^{8}+4\,{x}^{4}{y}^{4}+9\,{y}^{8}+
{x}^{12}+4\,{x}^{8}{y}^{4}+9\,{x}^{4}{y}^{8}+16\,{y}^{12} \right).
\end{eqnarray*}
\label{estimation-corollary}
\end{corollary}

We investigate the problem of rational solutions for the following optimization problem: given points $P_i=(x_i,y_i)$, $i=1,2,\cdots,N$, a small positive number $\delta>0$, and a set of monomials $m_j=x^{k_j}y^{l_j}$, $j=1,2,\cdots,q$, the objective is to find rational numbers $(c_1,c_2,\cdots,c_q)\in [-1,1]^q\subset \mathbb{Q}^q$ that satisfy the following conditions:

\begin{itemize}
\item The polynomial curve $f(x,y)=c_1m_1+c_2m_2+\cdots+c_qm_q=0$ passes through the line segment $[x_i,x_i]\times[y_i-\delta,y_i+\delta]$, for $i=1,2,\ldots,N$.
\item $\max{|c_j|, j=1,2,\cdots,q}=1$.
\item The sum of squares of $f(x_i,y_i)$, for $i=1,2,\ldots,N$, is minimized.
\end{itemize}

It is clear that if $(c_1,c_2,\cdots,c_q)$ is a local minimum, then $(-c_1,-c_2,\cdots,-c_q)$ is also a local minimum,
and $\min\{|c_j|, j=1,2,\cdots,q\}=1$.
This optimization problem can be decomposed into $q$ individual optimization problems $(Q_j)$, for $j=1,2,\cdots,q$.

\begin{equation}
(Q_j): \left\{
\begin{array}{rcl}
\min&& {\displaystyle \sum_{i=1}^N \left(c_1x_i^{k_1}y_i^{l_1}+c_2x_i^{k_2}y_i^{l_2}+ \cdots+c_qx_i^{k_q}y_i^{l_q}\right)^2,} \\
\mbox{s.t.}&& 
(c_1x_i^{k_1}\underline{y_i}^{l_1}+c_2x_i^{k_2}\underline{y_i}^{l_2}+ \cdots+c_qx_i^{k_q}\underline{y_i}^{l_q})\\
&&\phantom{x}
\times (c_1x_i^{k_1}\overline{y_i}^{l_1}+c_2x_i^{k_2}\overline{y_i}^{l_2}+ \cdots+c_qx_i^{k_q}\overline{y_i}^{l_q})<0,\\
&&\phantom{here}(\underline{y_i}=y_i-\delta, \;\overline{y_i}=y_i+\delta),\\
&& -1<c_1,c_2,\cdots,c_q<+1, \quad c_j=1. 
\end{array}
\right.
\label{min-qj}
\end{equation}

Each sub-problem $(Q_j)$ can be transformed into a linear inequality problem with constraints $c_1, c_2, \ldots, c_q$, where the objective function is a quadratic function. Therefore, it can be converted into a problem of solving a system of linear inequalities

If the coordinates of points $P_1, P_2, \ldots, P_N$ are all rational numbers, it is possible to obtain a rational solution to the original problem using symbolic computation methods. Subsequently, asymptotic analysis can be applied to find the "optimal" integer coefficient solution to the original problem.

\begin{example}\quad
Given any real number $b\,(0<b\leq 1)$, let $E(b):=\{(x,y)|x^2+y^2/b^2=1\}$ be an ellipse with $b$ as minor-axis and the major-axis equal to 1, and $L(b)$ be the maximal perimeter of triangles inscribed in the ellipse. The aim is to find an explicit formula that connects $L$ and $b$.

For a given value of $b$, the maximum inscribed triangle can be approximated using numerical methods to obtain an approximate value of $L(b)$. Let $\underline{L}(b)$ and $\overline{L}(b)$ be the lower and upper bounds of $L(b)$, respectively. By employing numerical methods, we can determine the upper and lower bounds to make the interval $[\underline{L}(b), \overline{L}(b)]$ arbitrarily small.

For rational numbers $b_i=i/16, i=1,2,\cdots,16$, we calculate the upper and lower bound intervals such that $\overline{L}(b_i)-\underline{L}(b_i)\leq 1/10^{14}$. This results in a dataset containing $16$ very short vertical line segments (see Figure 1). The dataset $S_i: [b_i,b_i]\times [\underline{L_i},\overline{L_i}]$ is as follows:

\begin{align*}
&S_1:[1/16, 4.00391006468634],&&S_2: [1/8, 4.01568603234052],\\
&S_3: [3/16, 4.03546516762171],&&S_4: [1/4, 4.06347582517416],\\
&S_5: [5/16, 4.10003596888417],&&S_6: [3/8, 4.14554927249083],\\
&S_7: [7/16, 4.20049672260437], &&S_8: [1/2, 4.26542082282451],\\
&S_9: [9/16, 4.34089883607274],&&S_{10}: [5/8, 4.42750157636573],\\
&S_{11}: [11/16, 4.52573604652852],&&S_{12}: [3/4, 4.63597477802543],\\
&S_{13}: [13/16, 4.75838211258069],&&S_{14}: [7/8, 4.89285523313257],\\
&S_{15}: [15/16, 5.03899987143207],&&S_{16}: [1, 5.19615242270664]. 
\end{align*}

Here, for simplicity, we use $[b_i, L_i]$ to represent the line segment $[b_i, L_i-\delta] \times [b_i, L_i-\delta]$, where $\delta=0.5\times 10^{-14}$. For example, $S_1:=[1/16, 4.00391006468634]$ represents the interval
$$
[1/16, 4.00391006468634-0.5\times 10^{-14}]\times 
[1/16, 4.00391006468634+0.5\times 10^{-14}].
$$

The current objective is to reconstruct $f(L, b) = 0$ using the $16$ vertical short lines and implicit function interpolation. From the references \cite{chen2014finding, tang2016resultant}, it is known that the solution to this problem is given by:
\begin{equation}
f(b,L)=
 \left( 1-{b^2} \right)^2 {L^4}-8\, \left( 2-{b^2} \right)
 \left( 1-2\,{b^2} \right)  \left( 1+{b^2} \right) {L^2}-432\,{b
^4}=0.
\label{fbl}
\end{equation}

\begin{figure}
\centering
\begin{picture}(350,260)(-20,-20) %%%(350,340)(-20,-20)
\put(0,0){\vector(1,0){340}}
\put(0,0){\vector(0,1){250}}
\put(345,-5){$b$}
\put(-5,255){$L$}
\put(32,0){\line(0,-1){3}}
\put(64,0){\line(0,-1){3}}
\put(96,0){\line(0,-1){3}}
\put(128,0){\line(0,-1){3}}
\put(160,0){\line(0,-1){5}}\put(155,-15){{\small 0.5}}
\put(192,0){\line(0,-1){3}}
\put(224,0){\line(0,-1){3}}
\put(256,0){\line(0,-1){3}}
\put(288,0){\line(0,-1){3}}
\put(320,0){\line(0,-1){5}}\put(318,-15){{\small 1}}
\put(-0,100){\line(-1,0){5}}\put(-22,95){{\small 4.5}}
\put(-0,200){\line(-1,0){5}}\put(-22,195){{\small 5.0}}
\thicklines
\put(20,0.7820140){\line(0,1){2}} %%%\put(10,0.3910070){\line(0,1){2}}
\put(40,3.1372080){\line(0,1){2}} %%%\put(20,1.5686040){\line(0,1){2}}
\put(60,7.0930340){\line(0,1){2}} %%%\put(30,3.5465170){\line(0,1){2}}
\put(80,12.6951660){\line(0,1){2}} %%%\put(40,6.3475830){\line(0,1){2}}
\put(100,20.0071940){\line(0,1){2}} %%%\put(50,10.0035970){\line(0,1){2}}
\put(120,29.1098560){\line(0,1){2}} %%%\put(60,14.5549280){\line(0,1){2}}
\put(140,40.0993460){\line(0,1){2}} %%%\put(70,20.0496730){\line(0,1){2}}
\put(160,53.0841660){\line(0,1){2}} %%%\put(80,26.5420830){\line(0,1){2}}
\put(180,69.179768){\line(0,1){2}} %%%\put(90,34.0898840){\line(0,1){2}}
\put(200,85.5003160){\line(0,1){2}} %%%\put(100,42.7501580){\line(0,1){2}}
%%%4.525736050, 4.635974780, 4.758382120, 4.892855240, 5.038999880, 5.196152430, 5.363443790, 5.539885490, 5.724454450, 5.916159900,
\put(220,105.1472100){\line(0,1){2}} %%%\put(110,52.5736050){\line(0,1){2}}
\put(240,127.1949560){\line(0,1){2}} %%%\put(120,63.5974780){\line(0,1){2}}
\put(260,151.6764240){\line(0,1){2}} %%%\put(130,75.8382120){\line(0,1){2}}
\put(280,178.5710480){\line(0,1){2}} %%%\put(140,89.2855240){\line(0,1){2}}
\put(300,207.7999760){\line(0,1){2}} %%%\put(150,103.8999880){\line(0,1){2}}
\put(320,239.230486){\line(0,1){2}} %%%\put(160,119.6152430){\line(0,1){2}}
\end{picture}
\caption{The 16 vertical segments determining a minimal polynomial $f(L,b)=0$.
The segment's length in this picture is enlarged $10^{9}$ times
for a better visualization effect.
}
\label{bl4}
\end{figure}
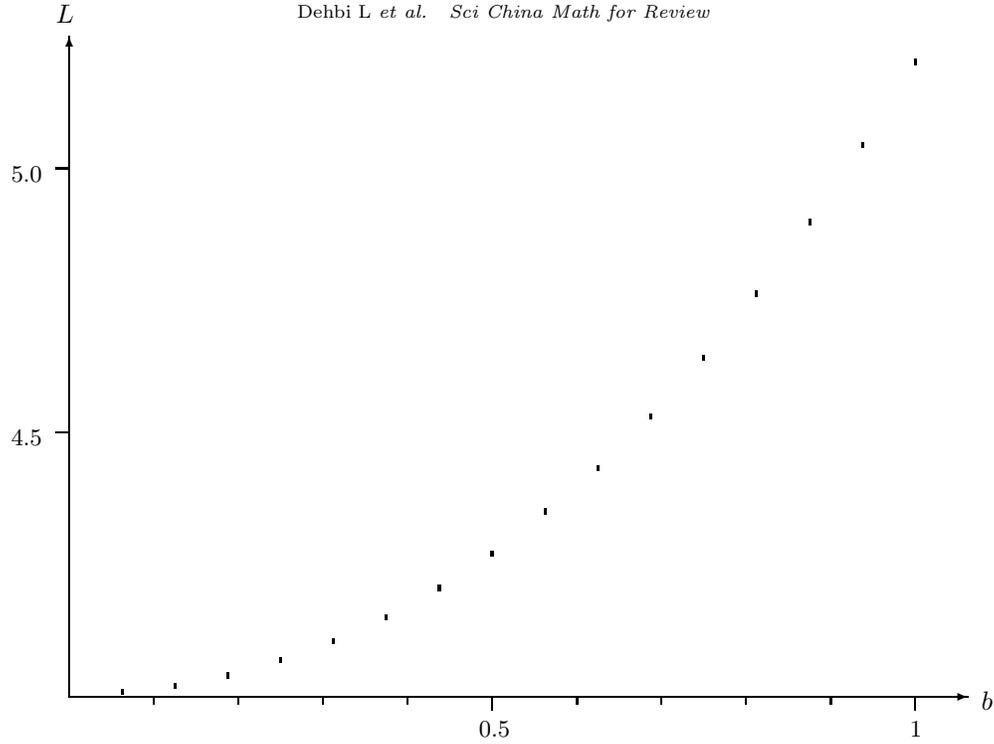

We will reconstruct this relationship through implicit function interval interpolation. The challenge lies in finding the most likely integer coefficient polynomial from the floating-point results.

Before interpolation, we first analyze the function between the perimeter $L$ of the maximum inscribed triangle and the semi-axis $b$ of the ellipse. This function can be represented by a polynomial equation $f(L,b)=0$.

To do so, let us assume that $ABC$ is the triangle inscribed in the ellipse $E(1,b): x^2+y^2/b^2=1$ with the maximum perimeter $L$. The coordinates of the vertices $A=(x_1,y_1)$, $B=(x_2,y_2)$, and $C=(x_3,y_3)$ of this maximum inscribed triangle are known. The side lengths of the triangle are denoted as $BC=u$, $CA=v$, and $AB=w$. We can establish the following equations:

\begin{eqnarray}
&f_0:=L-(u+v+w)=0, & f_4:=x_1^2+y_1^2/b^2-1=0,  \label{Lpqr}\\
&f_1:=(x_2-x_3)^2+(y_2-y_3)^2-u^2= 0,  & f_5:=x_2^2+y_2^2/b^2-1=0, \nonumber\\[10pt]
&f_2:=(x_3-x_1)^2+(y_3-y_1)^2-v^2 = 0,  & f_6:=x_3^2+y_3^2/b^2-1=0. \nonumber\\[10pt]
&f_3:=(x_1-x_2)^2+(y_1-y_2)^2-w^2 =0, & \nonumber
\end{eqnarray}
%\label{ellipse-triangle}
Then applying the Lagrange multiplier method we have
$$
\mathcal{L}=u+v+w-\sum_{j=1}^{6}k_j  f_j,
$$
and
\begin{eqnarray}
&&{\displaystyle f_7:=\frac{\partial \mathcal{L}}{\partial u}=0, \quad f_8:=\frac{\partial \mathcal{L}}{\partial v}=0,
\quad f_9:=\frac{\partial \mathcal{L}}{\partial w}=0,} \nonumber\\[10pt]
&&f_{10}:= {\displaystyle \frac{\partial \mathcal{L}}{\partial x_1}=0}, 
\quad f_{11}:= {\displaystyle \frac{\partial \mathcal{L}}{\partial x_2}=0},
\quad f_{12}:= {\displaystyle \frac{\partial \mathcal{L}}{\partial x_3}=0},\label{lpqr2}\\[10pt]
&&f_{13}:= {\displaystyle \frac{\partial \mathcal{L}}{\partial y_1}=0},
\quad f_{14}:= {\displaystyle \frac{\partial \mathcal{L}}{\partial y_2}=0},
\quad f_{15}:= {\displaystyle \frac{\partial \mathcal{L}}{\partial y_3}=0},\nonumber
\end{eqnarray}

The system of equations (\ref{Lpqr}) and (\ref{lpqr2}) consists of $16$ polynomial equations $f_i=0$ $(i=0,1,\ldots,15)$ and $17$ unknowns, namely:
$$
L,\; b,\; u,\; v,\; w, \; x_1,\;y_1,\;x_2,\;y_2;\;x_3,\; y_3,\;k_1,\;\cdots,\;k_6. 
$$ 

In theory, it is possible to eliminate $15$ variables from the system ${f_i=0, 1\leq i\leq q}$. For example, by eliminating all variables except $L$ and $b$, we obtain a polynomial in $L$ and $b$ denoted as $\mbox{\it Res/}(L, b) = 0$. Therefore, equation (\ref{fbl}) must be a factor of $\mbox{\it Res/}(L, b)$.

Previous work has shown that direct symbolic computation can eliminate $12$ variables from ${f_0, f_1, \ldots, f_{15}=0}$, specifically ${x_1, \ldots, y_6, k_1, \ldots, k_6}$, resulting in a system of $4$ equations only involving $L$, $b$, $u$, $v$, and $w$. In this reduced system, both $L$ and $b$ appear only with even powers. To save space, we denote these $4$ intermediate results as:

\begin{equation}
g_1(L^2,b^2,u,v,w)=0,\quad
g_2(L^2,b^2,u,v,w)=0,\quad
g_3(L^2,b^2,u,v,w)=0,\quad
g_4(L^2,b^2,u,v,w)=0.
\label{geqs}
\end{equation}

Further elimination of the variables $u$, $v$, and $w$ encounters storage space issues. To obtain the final result (\ref{fbl}), a combination of symbolic computation and sparse interpolation is required. In this paper, we attempt to use implicit function interval interpolation to construct a polynomial curve $F(x, y) = 0$ passing through the $16$ small line segments in Figure \ref{bl4}. We then prove that substituting $x=L$ and $y=b$ into $F(x, y)$ yields the solution to the original problem.

According to equation \eqref{geqs}, we are searching for curves of the following form that pass through the given 16 short lines:
\begin{equation}
F(x,y)=\,c_0+c_1x^2+c_2y^2+c_3 x^4+c_4x^2y^2+c_5y^4+\cdots
+c_{p-d-1}x^{2d}+\cdots+c_{p-1} y^{2d},
\label{pl2b2}
\end{equation}

Although the upper bound of $d$ can be estimated based on the degrees of $L,b,u,v,w$ in equation \eqref{geqs} and the form of the Sylvester resultant, the obtained value of $d$ is generally larger than the actual final result. Additionally, if the interpolation target polynomial has $d>5$, $F(x,y)$ contains 21 monomials:

$$
1;\quad
x^2,\, y^2; \quad
x^4,\, x^2y^2,\, y^4; \quad
x^6, \, x^4y^2, \, x^2y^4, \, y^6; 
$$
$$
x^8, \, x^6y^2, \, x^4y^4, \, x^2y^6, \, y^8;\quad
x^{10},\,x^8y^2,\,x^6y^4,\,x^4y^6,\,x^2y^8,\, y^{10},
$$
which exceeds the current number of data points ($16$). Therefore, we will try $d=2,3,4$ in ascending order. If $d=2$ is successful, we will then test the case of $d=1$. If $d=4$ is not successful, we will use numerical optimization to supplement more data and test the cases of $d=5, \ldots$. In the actual computations, we found that $d=2,3$ do not have polynomials of the form \eqref{pl2b2} that simultaneously pass through the 16 line segments in Figure \ref{bl4}, but there is such a polynomial curve for $d=4$. The specific calculation process is provided below.
%%%%%%%%%%%%%%%%%%%%%%%%%%%%%%%%%%%%%%%%%%%%%%%%%%%%%%%%%%%%%%%%%%%%%%%
%%%%%%%%%%%%%%%%%%%%%%%%%%%%%%%%%%%%%%%%%%%%%%%%%%%%%%%
According to Theorem~\ref{estimation} and Corollary~\ref{estimation-corollary}, we have the following result:
%\begin{proposition}\quad
%\label{estimation-r468}
\begin{enumerate}
%\item[ ]
\item If a polynomial
  $
  f(x,y)=c_0+c_1x^2+c_2y^2+c_3x^4+c_4x^2y^2+c_5 y^4
  $
  passes the line segments $S_1,S_2,\cdots,S_{16}$ and $c_0, c_1, \cdots, c_5\in [-1,1]$, then
  \begin{equation}
  \sum_{i=1}^{16} f(x_i,y_i)^2 \leq r_2:=1.2740\times 10^{-21};
  \label{estimation-r2}
  \end{equation}
\item If
  $
  f(x,y)=c_0+c_1x^2+c_2y^2+c_3x^4+c_4x^2y^2+c_5^4+c_5y^4+c_6x^6+c_7x^4y^2+c_8x^2y^4+c_9y^6
  $
  passes the line segments $S_1,S_2,\cdots,S_{16}$ and $c_0, c_1, \cdots, c_9\in [-1,1]$ , then
  \begin{equation}
  \sum_{i=1}^{16} f(x_i,y_i)^2 \leq r_3:=2.2799\times 10^{-18};
  \label{estimation-r3}
  \end{equation}
\item If
  $
  f(x,y)=c_0+c_1x^2+c_2y^2+\cdots+c_{10}x^8+c_{11}x^6y^2+\cdots+c_{14}y^8
  $
  passes the line segments $S_1,S_2,\cdots,S_{16}$ and $c_0, c_1, \cdots, c_{14}\in [-1,1]$, then
  \begin{equation}
  \sum_{i=1}^{16} f(x_i,y_i)^2 \leq r_4:=3.2326\times 10^{-15}.
  \label{estimation-r4}
  \end{equation}
  \end{enumerate}
%\end{proposition}
Let us denote the set of monomials as 
\begin{eqnarray*}
 &&{\cal P}_2:=\{1, x^2,y^2,x^4,x^2y^2, y^4\}, \\
&&{\cal P}_3:=\{1, x^2,y^2,x^4,x^2y^2, y^4, x^6, x^4y^2, x^2y^4, y^6\},\\
&&{\cal P}_4:=\{1, x^2,y^2,x^4,x^2y^2, y^4, x^6, x^4y^2, x^2y^4, y^6, x^8, x^6y^2, x^4y^4, x^2y^6, y^8\},
\end{eqnarray*}
and $P_i=(x_i,y_i)\,(i=1,2,\cdots,16)$ the barycenters of the segments $S_i$ as presented in Table \ref{tab4}.

\begin{table}[!h]
\centering
\caption{}{}The barycenters of the 16 segments
\label{tab4}
%%%
$$\begin{array}{|c|c|c||c|c|c|}
\hline
\phantom{PP}P_1\phantom{PP} &\phantom{x}1/16\phantom{x}&\phantom{x}4.003910064687\phantom{x} & \phantom{PP}P_9\phantom{PP}&\phantom{x}9/16\phantom{x}&\phantom{x} 4.340898836073\phantom{x}
\\ \hline
P_2&1/8 & 4.015686032341& P_{10}& 5/8& 4.427501576366
\\ \hline
P_3& 3/16& 4.035465167622   & P_{11}&11/16& 4.525736046529
\\ \hline
P_4 &1/4& 4.063475825175 &  P_{12}&3/4 & 4.635974778026
\\ \hline
P_5&5/16& 4.100035968885 & P_{13}&13/16& 4.758382112581
\\ \hline
 P_6&3/8& 4.145549272491   & P_{14}&7/8& 4.892855233133 \\ \hline
P_7&7/16& 4.200496722605 & P_{15}&15/16& 5.038999871433
\\ \hline
 P_8&1/2 & 4.265420822825 & P_{16}&{\small 1} & 5.196152422707 \\
\hline
\end{array}$$
%%%
\end{table}

For $d=2,3,4$ and $2\leq q\leq (d+1)(d+2)/2$, we define
$$
r(d,q):=\min_{|c_j|\leq 1, j=1,\cdots,q\atop
m_1, \cdots, m_q\in {\cal P}_d }
{\sum_{i=1}^{16} \left(\sum_{j=1}^q c_jm_j(x_i,y_i)\right)^2}.
$$
Then
$$
q_1 < q_2 \Longrightarrow r(d,q_1)\leq r(d,q_2).
$$

Let $r_2,r_3,r_4$ be the real numbers defined in \eqref{estimation-r2}, \eqref{estimation-r3}, and \eqref{estimation-r4}. It can be inferred from
$$
r(d,q)>\frac{q}{(d+1)(d+2)/2}\cdot r_d
$$
that there does not exist $m_1(x,y),\cdots, m_q(x,y)\in {\cal P}d$ and real numbers $c_1,\cdots, c_q$ such that
$
f(x,y)=c_1m_1(x,y)+\cdots+c_qm_q(x,y) (\not\equiv 0)
$
simultaneously passes through $S_1, S_2, \ldots, S{16}$.
By using symbolic computation, the following results can be obtained
$$
\scalebox{0.8}{$
\begin{array}{lclcl}
r(2,2)=0.4695\cdots,&&r(3,2)=0.2023,\cdots&&r(4,2)=0.1263\cdots,\\
r(2,3)=0.001129\cdots,&&r(3,3)=0.001129,\cdots&&r(4,3)=6.4239\times 10^{-4},\\
r(2,4)=1.0123\times 10^{-6},&&r(3,4)=1.2093\times 10^{-7},&&r(4,4)=1.2093\times 10^{-7},\\
r(2,5)=3.1611\times 10^{-8},&&r(3,5)=1.4273\times 10^{-8},&&r(4,5)=2.0731\times 10^{-9},\\
r(2,6)=3.9317\times 10^{-8},&&r(3,6)=4.2241\times 10^{-11},&& r(4,6)=6.9501\times 10^{-12},\\ \cline{1-1}
&& r(3,7)=1.5547\times 10^{-13}&& r(4,7)=1.5547\times 10^{-13},\\ \cline{5-5}
&&r(3,8)=6.6283\times 10^{-15},&&r(4,8)=4.939\times 10^{-30};\\
&&r(3,9)=5.0264\times 10^{-16},&&r(4,9)\leq 4.939\times 10^{-30}, \\ %%%2.9596\times 10^{-26}
&&r(3,10)=1.2133\times 10^{-16},&&r(4,10)\leq 4.939\times 10^{-30},\\ \cline{3-3} %%%2.9590\times 10^{-26}
&&&&r(4,11)\leq 4.939\times 10^{-30},\\ %%%2.9252\times 10^{-26}
&&&&r(4,12)\leq 4.939\times 10^{-30},\\ %%%2.2923\times 10^{-26}
&&&&r(4,13)\leq 4.939\times 10^{-30},\\ %%%2.4738\times 10^{-26}
&&&&r(4,14)\leq 4.939\times 10^{-30},\\ %%%3.76209\times 10^{-27}
&&&&r(4,15)\leq 4.939\times 10^{-30}.   %%%1.1946\times 10^{-27}
\end{array}
$}
$$
%%%%%%%%%%%%%%%%%%%%%%%%%%%%%%%%%%%%%%%%%%%%%%%%%%%%%%%%%%%%%%%%%%%%%%%
The above results show that when $d=2, 3$ and $d=4, q=8$, we have $r(d,q)>r_d$. This indicates that in order for the curve $F(x,y)=0$ in the form of equation \eqref{pl2b2} to pass through the 16 line segments, the degree of $F$ must be at least 8, and it must have at least 8 monomials.

Furthermore, we have:
$$
\min\sum_{i=1}^{16}\sum_{j=1}^{8}\left(c_jm_j(x_i,y_i)\right)^2
=\frac{\overbrace{249227\cdots358763}^{492 \text{digits}}}{\underbrace{754846\cdots\cdots047616}_{517 \text{digits}}}
\approx 3.301701\cdots\times10^{-26},
$$
and the coefficients of the monomials $
{x}^{4},{y}^{2},{y}^{4},{x}^{2}{y}^{2},{x}^{2}{y}^{4},{x}^{4}{y}^{2},
{x}^{4}{y}^{4},{x}^{6}{y}^{2},
$ can be expressed using continued fractions as follows:
$$
\scalebox{0.9}{$
c_1=1, \phantom{a large space here}
c_2=\cfrac{1}{26+\cfrac{1}{1+\cfrac{1}{3329579+\cdots}}},
\label{confracs}
$}
$$
$$
\scalebox{0.9}{$
c_3=-\cfrac{1}{431+\cfrac{1}{1+\cfrac{1}{208102+\cdots}}}, \phantom{space}
c_4=-\cfrac{1}{18+\cfrac{1}{14773286+\cfrac{1}{36+\cdots}}},
$}
$$
$$
\scalebox{0.9}{$
c_5=\cfrac{1}{216+\cfrac{1}{121148115+\cfrac{1}{10+\cdots}}}, \phantom{space}
c_6=-\cfrac{1}{18+\cfrac{1}{14773286+\cfrac{1}{1+\cdots}}},$}
$$
$$
\scalebox{0.9}{$
c_7=-\cfrac{1}{431+\cfrac{1}{1+\cfrac{1}{937254+\cdots}}}, \phantom{space} c_8 =\cfrac{1}{26+\cfrac{1}{1+\cfrac{1}{21212373+\cdots}}}.$}
$$
The corresponding polynomial curve is given by:
\begin{align}
f(x,y)=431.9999989\cdots{x}^{4}+ 16.00000013\cdots{y}^{2}
 -1.000000008\cdots{y}^{4}
 -23.99999985\cdots{x}^{2}{y}^{2} \nonumber\\
  +1.999999994\cdots{x}^{2}{y}^{4}
 -23.99999986\cdots{x}^{4}{y}^{2} -{x}^{4}{y}^{4}+ 15.99999998\cdots{x}^{6}{y}^{2},
\label{approximate-polynomial}
\end{align}

In this polynomial, the coefficients are multiplied by a certain factor so that ${\displaystyle \min_{1\leq j\leq q} |c_j| =1 }$, meaning the coefficient with the smallest absolute value is set to 1. Naturally, we conjecture that the corresponding integer coefficient polynomial for this polynomial is given by:

\begin{equation}
F(x,y)=432\,{x}^{4}+ 16\,{y}^{2}  -\,{y}^{4}-24\,{x}^{2}{y}^{2} +2\,{x}^{2}{y}^{4}
 -24\,{x}^{4}{y}^{2} -{x}^{4}{y}^{4}+ 16\,{x}^{6}{y}^{2}.
\label{polynomial88}
\end{equation}

There is a high possibility that this polynomial can also be expressed as a linear combination of the line segments $S_1, S_2, \ldots, S_{16}$. To confirm this conjecture, we only need to substitute the values $x_i=i/16, i=1,2,\ldots,16$ into \eqref{polynomial88} and verify, using numerical methods, that the signs of the resulting univariate polynomial equation $f(x_i,y)=0$ are opposite at the two endpoints of $S_i$. Computational calculations have indeed confirmed this.

Finally, based on the continued fraction representations of $c_1, c_2, \cdots, c_8$ (page \pageref{confracs}), we find a number $R$ such that, under the assumption that $c_1, c_2, \cdots, c_8 \in \mathbb{Z}$ and $|z_i| \leq R$, equation \eqref{polynomial88} is the unique eighth-degree curve that passes through $S_1, S_2, \cdots, S_{16}$. To do this, assume that $c_1, c_2, \cdots, c_8 \in \mathbb{Z}$ and $\max{|c_j|, j=1,2,\cdots,8}=R$. Then we have:

$$
\frac{|c_2|}{R}\in (\frac{1}{27}, \frac{1}{27}+\frac{1}{2427263793}),
$$
which implies,
$$
\frac{|c_2|}{R}-\frac{1}{27}\leq \frac{1}{2427263793},
$$
and further,
$$
\frac{1}{27R}\leq \frac{27|c_2|-R}{27R}\leq \frac{1}{2427263793}, ~~ \mbox{\it if\,}~ |c_2|/R\not=1/27.
$$
This means
$$ 
R\geq 89898659.
$$

After applying the same procedure to several other continued fractions, we have
\begin{align*}
R\geq \max\left\{89900495, 531838299/2, 130839964313/5, \right.\phantom{x}&\\
\phantom{xxx}\left. 265919167, 404894159, 572734097\right\}&\,>1.3083\times 10^{11},
\end{align*}
This indicates that under the condition $c_j<1.3088\times 10^{11},(j=1,2,\cdots,8)$, the polynomial \eqref{polynomial88} is the unique algebraic curve with integer coefficients of degree not exceeding $8$ that simultaneously passes through $S_1,S_2,\ldots,S_{16}$.

Note that the solutions we obtain using interval implicit function interpolation may only be a subset of the solutions to the original system of equations. Adding these solutions to the original system ${f_0, f_1, \ldots, f_{15}}$ can help us find additional solutions to the original system or prove that $F(L, b) = 0$ is the unique solution to the original system.

\end{example}
\section{Conclusion}
\label{conc}

This article discusses the problem of interval interpolation for planar algebraic curves. The input consists of several rectangular regions and vertical short lines on the plane. We discuss the solution methods for the following three problems and describe the solving process using several meaningful examples: (1) By simultaneously determining the lowest degree of the curves in the given small neighborhoods and constructing algorithms for curve construction, we transform the problem of determining the lowest degree into a positive-definiteness determination problem on a rectangular region in a high-dimensional space of multivariate polynomials. (2) By simultaneously constructing sparse solutions for algebraic curves in the given small neighborhoods, we study the analysis of observational data for the orbits of 38 celestial bodies in the solar system, including planets, asteroids, and dwarf planets. From this analysis, we rediscover Kepler's Third Law.
(3) By simultaneously constructing algorithms for integral coefficient algebraic curves in the given small neighborhoods, we explain the uniqueness analysis method for integral coefficient polynomial curves using an interesting set problem as an example.

The methods constructed in this article, as symbolic-numeric hybrid methods, help overcome the issue of intermediate process inflation that may occur in pure symbolic computation methods. They also assist in finding possible invariants and useful empirical formulas in data with errors. These methods are suitable for analyzing datasets obtained from practical observations with large sample sizes but limited accuracy (e.g., Example \ref{ex2} in this article), as well as handling datasets with high precision and small sample sizes generated for theoretical research problems (e.g., Example \ref{ex3}) or data consisting of accurate but non-rational numbers proposed in pure theoretical research (e.g., Example \ref{ex1}). The methods presented in this article can also be applied to algebraic surface interpolation problems in small neighborhoods in high-dimensional spaces.

%This paper studied the approximate implicit interpolation of plane algebraic curves $f(x, y) = 0$ that pass
%through a given set of points in the plane. Based on optimization problems representation and the Lagrange multipliers method, we
%constructed a square matrix $A^{[d]}$ of order $(d+1)(d+2)/2$ and proved that the lowest degree of polynomials
%that passes through the dataset is $d$ if and only if it is the minimal number satisfying $det A^{[d]} = 0$. The
%approach proved to be valid when the data are represented either as points or a set of rectangular regions in a given
%$R^2$, which refers to a dataset with errors. Moreover, we study the sparsity of the coefficients to determine
%the polynomial with the least number of monomials among all polynomials of the smallest degree that
%pass the given points or neighborhoods. We believe that the methods developed for plane algebraic
%curves in previous sections could be generalized to processing application data of engineering and optimal
%controls.

\Acknowledgements{This work was supported by the National Natural Science Foundation of China (Grant Nos. 2171159 and 12071282).}

\end{document}